\documentclass[11pt]{article}
\usepackage{amsmath,amsthm,amsfonts,amssymb,mathrsfs,bm,graphicx,stmaryrd,hyperref}
\usepackage{mathtools}
\usepackage[usenames]{color}
\usepackage[letterpaper,hmargin=1.0in,vmargin=1.0in]{geometry}
\parskip 	\smallskipamount

\usepackage{tikz}
\usetikzlibrary{arrows}
\usetikzlibrary{arrows.meta}
\usetikzlibrary{patterns}
\definecolor{uuuuuu}{rgb}{0.26666666666666666,0.26666666666666666,0.26666666666666666}
\definecolor{darkgreen}{HTML}{0d8513}

\newtheorem{theorem}{Theorem}[section]
\newtheorem{lemma}[theorem]{Lemma}
\newtheorem{corollary}[theorem]{Corollary}
\newtheorem{proposition}[theorem]{Proposition}

\newtheorem{remark}[theorem]{Remark}

\newtheorem{definition}[theorem]{Definition}
\newtheorem*{claim}{Claim}

\newtheorem{maintheorem}{Theorem}

\numberwithin{equation}{section}

\def\ind{{\mathbf 1}}
\def\N{\mathbb{N}}
\def\L{\mathbb{L}}
\def\P{\mathbb{P}}
\def\Z{\mathbb{Z}}
\def\R{\mathbb{R}}

\def\E{\mathbb{E}}
\def\l{\ell}

\newcommand{\cF}{\mathcal{F}}
\newcommand{\cD}{\mathcal{D}}

\newcommand{\cI}{\mathcal{I}}
\newcommand{\cT}{\mathcal{T}}
\newcommand{\cN}{\mathcal{N}}

\newcommand{\cE}{\mathcal{E}}

\newcommand{\e}{\varepsilon}

\newcommand{\boo}{\mathbf{0}}
\newcommand{\bn}{\mathbf{n}}
\newcommand{\us}{\underline{s}}
\newcommand{\os}{\overline{s}}
\newcommand{\uS}{\underline{S}}
\newcommand{\oS}{\overline{S}}
\newcommand{\uA}{\underline{A}}
\newcommand{\oA}{\overline{A}}
\newcommand{\oH}{\overline{H}}
\newcommand{\oM}{\overline{M}}
\newcommand{\ol}{\overline{l}}
\newcommand{\uT}{\underline{T}}
\newcommand{\oT}{\overline{T}}
\newcommand{\hS}{\hat{S}}
\newcommand{\ocT}{\overline{\mathcal{T}}}
\newcommand{\ucT}{\underline{\mathcal{T}}}
\newcommand{\ocG}{\overline{\mathcal{G}}}
\newcommand{\ucG}{\underline{\mathcal{G}}}
\newcommand{\hW}{\hat{W}}
\newcommand{\hxi}{\hat{\xi}}

\newcommand{\sT}{\mathscr{T}}
\newcommand{\sI}{\mathscr{I}}

\newcommand{\rd}{{\rm d}}

\usepackage{dsfont}
\newcommand{\don}{\mathds{1}}

\DeclareMathOperator{\MP}{MultiPeak}
\DeclareMathOperator{\LAP}{Laplace}
\DeclareMathOperator{\Exp}{Exp}

\begin{document}
\title{Infinite order phase transition in the slow bond TASEP}

\author{
Sourav Sarkar
\thanks{Centre for Mathematical Sciences, University of Cambridge.
Email:~ss2871@cam.ac.uk}
\and
Allan Sly
\thanks{Department of Mathematics, Princeton University. Email: allansly@princeton.edu}
\and
Lingfu Zhang
\thanks{Department of Mathematics, Princeton University, and Department of Statistics, University of California, Berkeley. Email: lfzhang@berkeley.edu}
}

\date{}
\maketitle

\begin{abstract}
In the slow bond problem  the rate of a single edge in the Totally Asymmetric Simple Exclusion Process (TASEP) is reduced from 1 to $1-\varepsilon$ for some small $\varepsilon>0$.   Janowsky and Lebowitz~\cite{JL92} posed the well-known question of whether such very small perturbations could affect the macroscopic current.  Different groups of physicists, using a range of heuristics and numerical simulations reached opposing conclusions on whether  the critical value of $\varepsilon$ is 0. This was ultimately resolved rigorously in~\cite{BSS14} which established that $\varepsilon_c=0$.

Here we study the effect of the current as $\varepsilon$ tends to 0 and in doing so explain why it was so challenging to predict on the basis of numerical simulations.  In particular we show that the current has an infinite order phase transition at 0, with the effect of the perturbation tending to 0 faster than any polynomial.  Our proof focuses on the Last Passage Percolation formulation of TASEP where a slow bond corresponds to reinforcing the diagonal.  We give a multiscale analysis to show that when $\varepsilon$ is small the effect of reinforcement remains small compared to the difference between optimal and near optimal geodesics.  Since geodesics can be perturbed on many different scales, we inductively bound the tails of the effect of reinforcement by controlling the number of near optimal geodesics and giving new tail estimates for the local time of (near) geodesics along the diagonal.
\end{abstract}

\section{Introduction}

In the Totally Asymmetric Simple Exclusion Process (TASEP) particles on $\Z$ move from left to right, jumping according to rate 1 Poisson clocks on each edge, but with moves blocked if there is a particle to its right.  Starting with work of Johansson~\cite{Jo99}, methods from integrable probability have provided an increasingly detailed description of the dynamics~\cite{BFPP,BL,BFS,BFS08,FP,L19,MQR}.  

The current, the rate at which particles cross the origin, is maximized at $\frac14$ when the particle density is one half.  This is a global property of the system, since conservation of particles means that the rate has to be equal everywhere.  Janowsky and Lebowitz~\cite{JL92} asked how this would be affected by a local perturbation, in particular reducing the rate of a single edge at the origin from 1 to $1-\varepsilon$, a so-called \emph{slow bond}.  Through heuristic arguments, they predicted a reduction in the current for any $\varepsilon>0$.  Later, another group of physicists arrived at the opposite conclusion~\cite{HTdN}.  Through numerical simulations and arguments of finite size scaling, they estimated $\varepsilon_c\approx 0.2$.  This problem remained unresolved until work of Basu et.~al.~\cite{BSS14} established that $\varepsilon_c = 0$, that is that there is always a slowdown.  The aim of this paper is to explain why this question was hard to answer heuristically and via simulations.

The movement of particles in TASEP can be mapped to Last Passage Percolation (LPP) with rate 1 exponential weights and it is this setup that we analyse.  The time for the $n$-th particle to pass the origin from step initial conditions is given by the passage time from the origin to $(n,n)$ which we denote $T_n$ (see Section~\ref{s:prelim} for the precise definitions).  The passage time grows like $\frac1{n}T_n\approx 4$ which corresponds to the inverse of the current.  In the slow bond model, the weights along the diagonal are said to be reinforced, replaced with larger rate $1-\e$ exponentials and we denote the corresponding passage time $T_n^\e$.  By the Subadditive Ergodic Theorem the limit satisfies a law of large numbers which we denote
\[
\Xi(\e) = \lim_{n\to \infty} \frac{1}{n}\E[T_n^{\varepsilon}] - 4.
\]
The results of~\cite{BSS14} show that $\Xi(\e)>0$ for all positive $\varepsilon$.  The proof is by a multi-scale argument which shows that the reinforcement  tends to attract the geodesic towards staying closer to the diagonal.  This happens if the optimal geodesic is close to another near optimal geodesic that spends more time along the diagonal.  After reinforcement, the new geodesic is larger.  This has a small probability for small $\varepsilon$, but improvements to the geodesic can be made on any scale so the proof makes use of an accumulation of increases to the expected passage time over a series of different scales, which in total show that $T_n^{\varepsilon}>4n$ for large $n$. Together with the Subadditive Ergodic Theorem this implies the result. No lower bound on this increase is given in~\cite{BSS14}, but when $\varepsilon$ is small the accumulation of many scales are needed and it would at best give a lower bound of  $\Xi(\e) \geq e^{-c\varepsilon^{-2}}$, which is very small indeed.  

One can ask whether this is an artifact of the proof or if $\Xi$ is indeed very small.  Here we show that it tends to 0 faster than any polynomial.

\begin{maintheorem}\label{t:main} 
For every $C\geq 1$ we have that  $\Xi(\varepsilon) \leq  O(\varepsilon^C)$.
\end{maintheorem}

This, in part, explains the difficulty in resolving the value of $\varepsilon_c$ numerically.  Tending to $0$ so quickly, it cannot be easily distinguished from a positive $\varepsilon_c$, particularly given $n$ of order ten thousand. 

We note the heuristics of Lebowitz~\cite{CLST} suggest that $\Xi(\varepsilon)$ is in fact of order $e^{-c\varepsilon^{-1}}$.
We expect that our method, which uses an induction-based multi-scale analysis (and is to be explained shortly), can be further pushed to yield $\Xi(\varepsilon)\leq e^{-(\log(\e^{-1}))^C}$ for some $C>1$; although for simplicity of the arguments we choose not to pursue that. However, to establish or refute the suggested order of $e^{-c\varepsilon^{-1}}$, some further ideas would be needed.

\subsection{Proof Sketch}  \label{sec:psk}

In Exponential LPP, without reinforcement along the diagonal the transversal fluctuations of the geodesic scale like $n^{2/3}$ (\cite{Jo00trans}).  It is, therefore, natural to expect then that the optimal geodesic spends about $n^{1/3}$ time on the diagonal.  We establish such a local time result  together with exponential tail bounds.  Thus, reinforcing the diagonal increases the original geodesic by order $\e n^{1/3}$.  The fluctuations in the passage time themselves are also of order $n^{1/3}$ and Tracy-Widom distributed and so are of the same magnitude.  Our proof rests on a comparison between the benefit of reinforcement accumulated on a series of smaller scales and the difference in passage times between geodesics and near geodesics.

On a rectangle, parallel to the diagonal, of length $n$ and height $n^{2/3}$, there are $n^{4/3}$ geodesics joining pairs of points on the left and right sides. These have a strong tendency to coalesce together forming highways from left to right; and in the middle third of the rectangle only $O(1)$ distinct geodesics remain which was established in~\cite{BHS} to show that there are no non-trivial infinite bigeodesics.  The proof makes use of the geometric fact that geodesics cannot cross each other twice and so a large number of distinct geodesics implies the existence of many good non-crossing disjoint paths.  Many such paths can be ruled out using a combination of the BK inequality with an entropy argument controlling the number of non-overlapping paths.

Since there are few highways and their placement is random, they are unlikely to spend much time close to the diagonal and benefit from reinforcement.  However, reinforcement may make another route preferable so we also need to consider the locations of \emph{near geodesics}, i.e., up-right paths close in passage time to the optimal ones.  We cannot perform the same geometric reduction since near geodesics can cross each other multiple times.  Instead, we discretize space and for each pair of starting and ending points rank the near geodesics.  The rank-$k$ best geodesics for each pair of starting and ending points have the property that they cannot cross twice and so we can again apply the approach of~\cite{BHS}.

This reduces our task to showing that, at an appropriate level of discretization, between a pair of starting and ending points, there are few near geodesics.  If we consider the passage time of the best geodesic that passes through a given point on an anti-diagonal parallel to the side of the rectangle, this scales asymptotically to a sum of two Airy processes and is thus locally Brownian.  As with Brownian motion, it is unlikely to have many well separated almost-maxima on an interval. By the Robinson-Schensted-Knuth (RSK) correspondence, this process is in fact a sum of two random walk bridges, conditioned to lie above another stochastic process (see e.g. \cite{DNV,O03} for this in slightly different settings).  Using random walk estimates and a comparison with Brownian motion we control the tails of the number of near-maxima.  

Altogether we show that there are only $O(1)$ near geodesics in the middle portion of the rectangle.  By considering translations of the field (which also translate the geodesics), we argue that in expectation these near geodesics spend only $n^{1/3}$ time close to the diagonal.  Moreover, the locations are essentially local and we establish enough independence to show corresponding tail bounds via a multi-scale proof.

The final step in the proof of Theorem~\ref{t:main} is a multi-scale induction controlling the increase in the passage time from reinforcement.  On a series of scales $n_k=\e^{-k/200}$ we show that with very high probability, the increase from reinforcement is at most $\e^{1/3} n_k^{1/3} (\log(n_k))^{50k}$.  Then at scale $k+1$, the inductive hypothesis shows that after reinforcement the new geodesic must be close to an original near geodesics.  Having shown that the  original near geodesics do not spend too much time close to the diagonal we can control the increase in the passage time from reinforcement at level $k+1$.  This induction will fail for some large enough $k$ (the slow bond results of~\cite{BSS14} guarantee this) but we show that for each fixed $k$ it will   hold provided $\e$ is small enough.  The bounds can then be used to establish Theorem~\ref{t:main} directly.

\subsection*{Organisation of the paper}
The rest of the paper is organised as follows. In Section \ref{s:prelim} we setup notations for the model of  Exponential LPP, and record some useful results (of the unperturbed LPP) from the literature. 
In Section \ref{s:loc} we prove the estimate of the time spent by geodesics near the diagonal. There we use an estimate on the number of near geodesics between a pair of points, whose proof is the main content of Section \ref{s:multi}.
The multi-scale proof of the main result is given in Section \ref{s:main}.

\subsection*{Acknowledgements}
The authors thank Riddhipratim Basu for many useful discussions.
The authors would also like to thank the anonymous referees for carefully reading this manuscript and providing valuable feedback that helped improve the exposition.
AS was supported by NSF grants DMS-1855527 and DMS-1749103, a Simons Investigator grant, and a MacArthur Fellowship.

\section{Notation and Preliminaries}  \label{s:prelim}

We now formally setup the model of Exponential LPP.
To each vertex $v \in \Z^2$ we associate an independent weight $\xi(v)$ with $\Exp(1)$ distribution.
For two points $u, v \in \Z^2$, we say $u\preceq v$ if $u$ is coordinate-wise less than or equal to $v$.
For such $u, v$ and any up-right path $\gamma$ from $u$ to $v$, we define the \emph{passage time of the path} to be
\[
T(\gamma) := \sum_{w \in \gamma} \xi(w) .
\]
Then almost surely there is a unique up-right path from $u$ to $v$ that has the largest passage time.
We call this path the \emph{geodesic} $\Gamma_{u,v}$, and call $T_{u,v}:=T(\Gamma_{u,v})$ the \emph{passage time from $u$ to $v$}.
An up-right path $\gamma$ from $u$ to $v$ is called an \emph{$x$-near geodesic}, if $T(\gamma)\ge T_{u,v}-x$.
For each $n\in \Z$ we denote $\bn=(n,n)\in\Z^2$; in particular we have $\boo=(0,0)$.
For $n\in \N$ we let $T_n=T_{\boo,\bn}$.

The TASEP on $\Z$ is mapped to this Exponential LPP in the following way. 
Consider the step initial condition, where each non-positive site is occupied by a particle, and each positive site is empty.
For each $x, y \in \Z_{\ge 0}$, we let $\xi((x,y))$ be the waiting time for the particle initially at site $-x$ to make its $y+1$-th jump, after it has made the previous jump and the site right next to it (which is site $y-x+1$) becomes empty.
Note that these waiting times would be i.i.d. $\Exp(1)$, from the model definition of TASEP.
Then via a recursive relation, $T_{\boo,(x,y)}$ would be the total time till the particle initially at site $-x$ makes the $y+1$-th jump. 

In the slow bond model, each jump cross the edge $0-1$ is done at a slower rate of $1-\e$.
Thus for this perturbed model, the corresponding LPP would be on the field $\{\xi^\e(v)\}_{v\in\Z^2}$, where all the weights $\xi^\e(v)$ are independent, and $\xi^\e((x,y))$ has distribution $\Exp(1)$ when $x\neq y$, or $\Exp(1-\e)$ (with rate $1-\e$ and mean $(1-\e)^{-1}>1$) when $x=y$.
We shall call this line $x=y$ \emph{the diagonal}.
We couple $\{\xi^\e(v)\}_{v\in\Z^2}$ with $\{\xi(v)\}_{v\in\Z^2}$ as follows. For each $v$ not on the diagonal, we let $\xi^\e(v)=\xi(v)$; and for each $v$ on the diagonal, we let $\xi^\e(v)=\xi(v)+\varrho(v)\xi'(v)$, where $\varrho(v)$ is a Bernoulli$(\e)$ random variable and $\xi'(v)$ is an $\Exp(1-\e)$ random variable with mean $(1-\e)^{-1}=1+\Theta(\e)$, independent of each other and both are independent of $\xi(v)$.

For points $u\preceq v$, we let $\Gamma^\e_{u,v}$ and $T^\e_{u,v}$ be the geodesic and passage time from $u$ to $v$, under the reinforced field $\{\xi^\e(w)\}_{w\in\Z^2}$.
For $n\in \N$ we also denote $T^\e_n=T^\e_{\boo,\bn}$.

We also set up the following useful notations.
For any $u=(x,y)\in \Z^2$, we denote $d(u)=x+y$, and $ad(u)=x-y$.
For each $n\in\Z$ we denote $\L_n=\{u\in\Z^2: d(u)=n\}$.
We shall use the notation $\llbracket \cdot, \cdot \rrbracket$ to denote discrete intervals, i.e., $\llbracket a,b \rrbracket$ will denote $[a,b]\cap \Z$. 

For merely avoiding the notational overhead of integer parts, we would ignore some rounding issues.
For example, we shall often assume, without loss of generality, that fractional powers of integers i.e., $k^{2/3}$ or rational multiples of integers as integers themselves. It is easy to check that such assumptions do not affect the proofs in any substantial way.

\subsection{Results on the unperturbed LPP}

We record some useful results from the literature, on the LPP on the original i.i.d. $\Exp(1)$ field.

For the last passage time from $u$ to $v$, we have the following one point estimates by the connection with random matrices.
For any $m,n \in \N$, $T_{\mathbf{0},(m,n)}$ has the same law as the largest eigenvalue of $X^*X$ where $X$ is an $(m+1)\times (n+1)$ matrix of i.i.d.\ standard complex Gaussian entries (see e.g. \cite[Proposition 1.4]{Jo99}).
Using this we get the following one point estimates from \cite[Theorem 2]{LR10}. 

Here and for the rest of the text, for each $m,n\in\Z_{\ge 0}$ we denote $D_{(m,n)}=(\sqrt{m}+\sqrt{n})^2$.
\begin{theorem}
\label{t:onepoint}
For each $\psi>1$, there exist $C,c>0$ depending on $\psi$ such that for all $m,n\in \N$ with $\psi^{-1}<\frac{m}{n}< \psi$ and all $x>0$ we have:
\begin{enumerate}
\item[(i)] $\P[T_{\mathbf{0}, (m,n)}-D_{(m,n)} \geq xn^{1/3}] \leq Ce^{-c\min\{x^{3/2},xn^{1/3}\}}$.
\item[(ii)] $\P[T_{\mathbf{0}, (m,n)}-D_{(m,n)} \leq -xn^{1/3}] \leq Ce^{-cx^3}$.
\item[(iii)] $|\E[T_{\mathbf{0}, (m,n)}] -D_{(m,n)}|\leq Cn^{1/3}$.
\end{enumerate}
For any $m\geq n\in \N$ and for all $x>0$, we have
\begin{enumerate}
\item[(iv)] $\P[T_{\mathbf{0}, (m,n)}-D_{(m,n)} \geq xm^{1/2}n^{-1/6}] \leq Ce^{-cx}$.
\end{enumerate}
\end{theorem}
We shall next quote a result about last passage times across parallelograms. These were proved in \cite{BSS14} for Poissonian LPP (see Proposition 10.1, Proposition 10.5 and Proposition 12.2 in \cite{BSS14}). A proof for the exponential setting can be found in 
\cite[Appendix C]{basu2019temporal}.

Consider the parallelogram $U$ whose one pair of sides lie on $\L_0$ and $\L_{2n}$ with length $2n^{2/3}$ and midpoints $(mn^{2/3},-mn^{2/3})$ and $(n,n)$ respectively. 
Let $U_1$ (resp.\ $U_2$) denote the intersections of $U$ with the strips $\{u:0\leq d(u)\leq 2n/3\}$ and $\{u:4n/3\leq d(u)\leq 2n\}$ respectively.
\begin{theorem}
\label{t:supinf}
For each $\psi<1$, there exists $C,c>0$ depending only on $\psi$ such that for all $|m|<\psi n^{1/3}$ and $U$ as above we have
\begin{enumerate}
\item[(i)] 
for all $x, L>0$ and $n$ sufficiently large depending on $L$,
$$\P\left[ \inf_{u,v\in U: d(v)-d(u)\geq \frac{n}{L}}  (T_{u,v}-D_{v-u}) \leq -xn^{1/3}\right]\leq Ce^{-cx^3}.$$
\item[(ii)]
for all $x>0$ and $n\geq 1$,
$$\P\left[ \sup_{u\in U_1,v\in U_2}  (T_{u,v}-D_{v-u}) \geq xn^{1/3}\right]\leq Ce^{-c\min\{x^{3/2},xn^{1/3}\}}.$$
\end{enumerate}
\end{theorem}

For each $n\in \N$, consider the passage times from $\boo$ to the line $\L_n$.
This point-to-line profile is known to have a scaling limit being the Airy$_2$ process, which locally looks like Brownian motion with modulus of continuity in the square root order.
Below is a quantitative estimate on the continuity of the point-to-line profile. Similar results have appeared as \cite[Theorem 3]{BG21} and also in \cite{Ham16} in the setting of Brownian last passage percolation.
\begin{lemma}  \label{lem:lpp-devi}
For any $\psi\in(0,1)$ there are $c,C>0$ such that the following is true.
For $n,a,b\in\N$ with $\psi n < a< b < (1-\psi)n$, there is 
\[
\P[|(T_{\boo,(a,n-a)} - D_{(a,n-a)} )-(T_{\boo,(b,n-b)}- D_{(b,n-b)})| > (\log(n))^7 \sqrt{b-a}] < Ce^{-c(\log(n))^2}
\]
\end{lemma}
We remark that the exponents of the logarithm factors are not sharp, but would suffice for our use cases.
To prove this, we need an estimate on the fluctuation of geodesics, near the end points.
\begin{lemma}   \label{lemma:trans-fluc-one-pt}
For each $\psi\in(0,1)$, there exist constants $C,c>0$ such that the following is true.
For any $n'<n \in \N$ large enough, $b, b',b^*\in \Z$, such that $\psi n<b<(1-\psi)n$, and $b'$ is the largest integer with $0<\frac{b'}{n'-b'}\le\frac{b}{n-b}$, and $(b^*,n'-b^*)$ is the intersection of $\Gamma_{\boo,(b,n-b)}$ with $\L_{n'}$.
Then we have $\P[|b'-b^*|>x{n'}^{2/3}] < Ce^{-cx}$ for any $x>0$.
\end{lemma}
This is a slight generalization of \cite[Proposition 2.3]{zhang2019optimal}, as it is for geodesics in any direction bounded away from the axis directions (whereas \cite[Proposition 2.3]{zhang2019optimal} only concerns geodesics in the $(1,1)$ direction); but the proofs are essentially verbatim, so we omit the details here.
See also \cite[Theorem 3]{basu2019coalescence} for the same result in a slightly different setting.

\begin{proof}[Proof of Lemma \ref{lem:lpp-devi}]
We shall again let $c,C>0$ denote small and large constants depending on $\psi$, and the values can change from line to line.
We also assume that $n$ is large enough and $b-a<(\log(n))^{-5}n^{2/3}$, since otherwise the statement holds obviously or by Theorem \ref{t:onepoint}.
Let $n'=n-\lfloor (\log(n))^7(b-a)^{3/2} \rfloor$. Let $(a^*,n'-a^*)$ be the intersection of $\Gamma_{\boo,(a,n-a)}$ with $\L_{n'}$, and $a'$ be the largest integer  with $0<\frac{a'}{n'-a'}\le \frac{a}{n-a}$.

The general idea is to use Lemma \ref{lemma:trans-fluc-one-pt} to bound $|a'-a^*|$. Then we use the fact that $T_{\boo,(b,n-b)} \ge T_{\boo,(a^*,n'-a^*)} + T_{(a^*+1,n'-a^*),(b,n-b)}$ and  Theorem \ref{t:onepoint} to upper bound $T_{\boo,(a,n-a)} - T_{\boo,(b,n-b)}$.
The lower bound follows similarly.

By Lemma \ref{lemma:trans-fluc-one-pt}, we have $\P[|a'-a^*|>(\log(n))^2(n-n')^{2/3}]<Ce^{-c(\log(n))^2}$.
On the other hand, consider the event $\cE_{devi}'$, where \[|T_{(i+1,n'-i),(b,n-b)}-D_{(b-i-1,n-b-n'+i)}|, |T_{(i+1,n'-i),(a,n-a)}-D_{(a-i-1,n-a-n'+i)}| \le (\log(n))^2(n-n')^{1/3}\]
for any $i\in\Z$ such that $c<\frac{b-i-1}{n-b-n'+i}, \frac{a-i-1}{n-a-n'+i}<C$ (note that this condition holds when $|i-a'|\le (\log(n))^2(n-n')^{2/3}+1$).
By applying Theorem \ref{t:onepoint} (i) and (ii) to each $|T_{(i+1,n'-i),(b,n-b)}-D_{(b-i-1,n-b-n'+i)}|$ and  $|T_{(i+1,n'-i),(a,n-a)}-D_{(a-i-1,n-a-n'+i)}|$, and taking a union bound, we have $\P[\cE_{devi}']>1-Ce^{-c(\log(n))^2}$.

When $|a'-a^*|\le (\log(n))^2(n-n')^{2/3}$ and $\cE_{devi}'$ holds, we have
\[
\begin{split}
& T_{\boo,(a,n-a)}-T_{\boo,(b,n-b)} 
\\
\le & T_{(a^*+1,n-a^*),(a,n-a)} \vee T_{(a^*,n-a^*+1),(a,n-a)} - T_{(a^*+1,n-a^*),(b,n-b)}
\\
\le & 2(\log(n))^2(n-n')^{1/3} + \max_{|i-a'|\le (\log(n))^2(n-n')^{2/3}+1} D_{(b-i-1,n-b-n'+i)} - D_{(a-i-1,n-a-n'+i)}\\
<& 
C(\log(n))^4(n-n')^{1/3} + D_{(b,n-b)} - D_{(a,n-a)}.
\end{split}
\]
Thus we have
\[
\P[(T_{\boo,(b,n-b)}-D_{(b,n-b)})-(T_{\boo,(a,n-a)} - D_{(a,n-a)}) > (\log(n))^7\sqrt{b-a}]<Ce^{-c(\log(n))^2}.
\]
Similarly the same inequality holds when exchanging $a$ and $b$, and then the conclusion follows.
\end{proof}
We also need the following estimate on the transversal fluctuation of near geodesics.
Let $A$ be the segment lying on $\L_0$ with length $2n^{2/3}$ and midpoint $(mn^{2/3}, -mn^{2/3})$, and let $B$ be the segment lying on $\L_{2n}$ with length $2n^{2/3}$ and midpoint $(n, n)$.
Let $U_{m,\phi}$ be the parallelogram, with one pair of sides lie on $\L_0$ and $\L_{2n}$ with length $2\phi n^{2/3}$, and midpoints $(mn^{2/3}, -mn^{2/3})$ and $(n,n)$.
\begin{proposition}[\protect{\cite[Proposition C.8]{basu2019temporal}}]   \label{prop:trans-fluc}
For each $\psi\in(0,1)$, there exists a constant $c>0$ such that the following is true.
Consider the event where there is an up-right path $\gamma$ from some $u\in A$ to $v\in B$, such that $\gamma$ is not contained in $U_{m,\phi}$, and $T(\gamma)>\E[T_{u,v}]-c\phi^2 n^{1/3}$.
Then the probability of this event is at most $e^{-c\phi^3}$,
if $|m| < \psi n^{1/3}$ and $\phi$ is large enough.
\end{proposition}

We would use the following result on the number of disjoint near geodesics. 
It is a an extension (from geodesics to near geodesics) of \cite[Proposition 3.1]{BHS}, while the proof is essentially verbatim, using the BK inequality combined with entropy estimates.

Let $A_k$ be the segment lying on $\L_0$ with length $2k^{1/16}n^{2/3}$ and midpoint $(mn^{2/3}, -mn^{2/3})$, and let $B_k$ be the segment lying on $\L_{2n}$ with length $2k^{1/16}n^{2/3}$ and midpoint $(n, n)$.
\begin{theorem}  \label{thm:disjoint-near-geo}
For each $\psi\in(0,1)$, there exist constants $c, n_0, N_0>0$ such that the following is true for any $n,N\in \N$, $n>n_0$, $n^{0.01}>N>N_0$, and $m$ with $|m|+N^{1/8}<\psi n^{1/3}$. 

Consider the event where there are $N$ $n^{1/3}$-near geodesics from $A_N$ to $B_N$, that are mutually disjoint.
Then the probability of this event is at most $e^{-cN^{1/4}}$.
\end{theorem}

\section{Tails of local times of near geodesics } 
\label{s:loc}
In this section we bound the time that near geodesics spend near the diagonal.

For any $w\in \N$, we consider barriers of length $2w$, each centered and perpendicular to the diagonal, and are spaced $w^{3/2}$ apart. That is, they are open line segments joining $(iw^{3/2}-w,iw^{3/2}+w)$ and $(iw^{3/2}+w,iw^{3/2}-w)$.
More precisely, for each $i\in\Z$, we let $B_i=\{(iw^{3/2}-x,iw^{3/2}+x): |x|<w\}$.
We shall estimate the number of barriers that can be hit by a near geodesic.

For any $u\preceq v$ and $x>0$, we let $H^x_{u,v}(w)=\max_{\gamma}|\{i\in\Z: B_i\cap \gamma \neq\emptyset\}|$,
where the maximum is over all $x$-near geodesic $\gamma$ from $u$ to $v$, i.e., over all up-right path $\gamma$ from $u$ to $v$ with $T(\gamma)\ge T_{u,v}-x$.

For any real numbers $p<q$, we denote
\[
H^x(p,q,w)=\max\{H^x_{u,v}(w): i,j\in\Z; p\le iw^{3/2} < jw^{3/2} \le q; u\in B_i, v\in B_j\} ,
\]
and we let $H^x(n,w)=H^x(0,n,w)$ for any $n\in\N$.

\begin{theorem}   \label{thm:near-loc-time}
There exist constants $C,c>0$, such that
\[
\P[H^x(n,w) > M(\log(w))^{32}n^{1/3}w^{-1/2}+1] < Ce^{-cM}
\]
for any $M>0$, $n^{0.99}<w^{3/2}$, and $0\le x\le w^{1/2}$.
\end{theorem}
\begin{remark}
A special case of this theorem is when $x=0$, i.e., we can bound the maximum number of barriers hit by a geodesic (rather than a near geodesic).
In this case the condition $n^{0.99}<w^{3/2}$ is not necessary, since the only reason to have this condition is that we want to apply Proposition \ref{prop:multiple-peak} below to bound the number of near geodesics between a pair of points; while for geodesics there is only one between any pair of points.
More precisely, for the case of $x=0$ and $w=1$, it can be shown that (see \cite[Lemma 2.18]{GZloc}),
\begin{equation}  \label{eq:igaabM}
\P\left[ \max_{a, b\in\llbracket 0,n\rrbracket, a<b}| \{i\in \Z: \mathbf{i} \in \Gamma_{\mathbf{a},\mathbf{b}}\}| > Mn^{1/3} \right] < Ce^{-cM},    
\end{equation}
for $C,c>0$ being constants and any $M>0$.
A proof of this estimate can be found in \cite[Appendix A]{GZloc}, which uses the same method as the proof of Theorem \ref{thm:near-loc-time}.

Such estimates on the time that geodesics spend on a diagonal are useful in studying geometries of geodesics. For example, in \cite{GZloc,GZfract} such results are used to construct and prove convergence to geodesic local times in the directed landscape, a scaling limit of the Exponential LPP \cite{DOV,DV21}.
\end{remark}
\begin{remark}
It would also be interesting to get a lower bound for the number of barriers hit by a geodesic.
A lower bound in expectation can be quickly deduced from the following estimate.
\begin{proposition}[\protect{\cite[Theorem 2]{basu2021small}}] \label{prop:BB212}
For any $\epsilon>0$, there exists a constant $c>0$, such that for any large enough $n$, and $n^{-2/3}\le \delta\le 1$, and $t\in\llbracket \epsilon n, (1-\epsilon)n\rrbracket$, we have
\[
\P[|m_*|< \delta n^{2/3}] > c\delta,
\]
where $m_*\in \Z$ is the number such that $(t-m_*, t+m_*)\in \Gamma_{\boo,\bn}$.
\end{proposition}
We note that \cite[Theorem 2]{basu2021small} is stated to require $n$ to be large enough depending on $\delta$; but the proof in \cite{basu2021small} actually allows for more general choices of $\delta$ and $n$, as discussed below the statement there.
Then for any $n$ large enough and $w\le n^{2/3}$, by taking $\delta = wn^{-2/3}$ and summing over order $nw^{-3/2}$ many choices of $t$, we get
\begin{equation}\label{eq:EH-lb}
    \E[| \{i\in \Z: B_i\cap \Gamma_{\boo, \bn}\neq \emptyset\}|]>cn^{1/3}w^{-1/2},
\end{equation}
where $c>0$ is a constant. In addition, \eqref{eq:igaabM} and \eqref{eq:EH-lb} together imply the following: for any small $\delta>0$, there is $c(\delta)>0$ such that $\P[| \{i\in \Z: \mathbf{i} \in \Gamma_{\boo, \bn}\}| > \delta n^{1/3}] > c(\delta)$ for any large enough $n$.

As for near geodesics, we expect a generalization of \eqref{eq:EH-lb} as stated below to hold:
\[
\E[| \{i\in \Z: B_i\cap \gamma\neq \emptyset, \forall \gamma \text{ up-right path from }\boo\text{ to }\bn,\; T(\gamma)\ge T_{\boo,\bn}-w^{1/2} \}|]>cn^{1/3}w^{-1/2}.
\]
To prove this, it suffices to upgrade Proposition \ref{prop:BB212} to (using the notations there):
\[
\P\Big[|m_*|< \delta n^{2/3}, \sup_{|m|\ge\delta n^{2/3}} T_{\boo,(t-m,t+m)}+T_{(t-m,t+m),\bn} -\xi((t-m, t+m)) < T_{\boo,\bn}-\delta^{1/2}n^{1/3}\Big] > c\delta.
\]
To establish this upgrade of Proposition \ref{prop:BB212}, a potential route is to analyze the point-to-line profiles $m\mapsto T_{\boo,(t-m,t+m)}$ and $m\mapsto T_{(t-m,t+m),\bn}-\xi((t-m, t+m))$, using the random walk Gibbs property (stated as Lemma \ref{lem:T-Gibbs} below). We do not pursue a rigorous proof of this upgrade here.

We also mention that the convergence results in \cite{GZloc} also contain a lower bound, for the setting of the $(1,1)$-direction semi-infinite geodesic, which is an infinite up-right $\Z^2$ path $\Gamma_\boo=\{\boo=(x_0, y_0), (x_1, y_1),\ldots \}$, satisfying (1) $(x_i,y_i)-(x_{i-1},y_{i-1})\in \{(0,1), (1,0)\}$ for each $i\in\N$; (2) $\lim_{i\to\infty} x_i/y_i=1$; (3) for any $i<j$, $\Gamma_{(x_i,y_i), (x_j,y_j)}$ is contained in $\Gamma_\boo$.
The existence and almost sure uniqueness of $\Gamma_\boo$ have been established in the literature (see \cite{Cou11, FP05}).
The main result of \cite{GZloc} implies that as $n\to\infty$, $2^{2/3}n^{-1/3}| \{i\in \llbracket 0,n\rrbracket: \mathbf{i} \in \Gamma_{\boo}\}|$ converges to a random variable $L(1)$, where $L:[0,\infty)\to \R$ is the directed landscape geodesic local time constructed in \cite{GZfract}.
It is shown (as \cite[Proposition 7.6]{GZfract}) that $L(1)>0$ almost surely.
\end{remark}

We can extend Theorem \ref{thm:near-loc-time} by taking wider barriers, and the next result is what will be used in the next section.

For each $h\ge 1$ and $i\in \Z$, let $hB_i$ be the barrier with the same center as $B_i$, but length $2hw$; i.e. we let $hB_i=\{(iw^{3/2}-x,iw^{3/2}+x): |x|<hw\}$.
For any $n\in\N$ and $x>0$, we let $H^x(n,w;h)$ be defined the same way as $H^x(n,w)$, except for replacing each $B_i$ by $hB_i$.
\begin{theorem}  \label{thm:near-geo-loc}
There exist constants $C,c>0$, such that $\P[H^x(n,w;h) > Mh(\log(w))^{32}n^{1/3}w^{-1/2}+1] < Ch^{3/2}e^{-cM}$ for any $M>0$, $h\ge 1$, $n^{0.99}<w^{3/2}$, $x\le w^{1/2}$.
\end{theorem}
\begin{proof}
We can assume that $h<w^3$ and $n^{2/3}> hw$, since otherwise $h(\log(w))^{32}n^{1/3}w^{-1/2} \ge nw^{-3/2}$, and the conclusion follows since there is always $H^x(n,w;h) \le nw^{-3/2}+1$.

For each $k=0,\ldots, h^{3/2}-1$, we let $H_k$ be the maximum number of barriers in $\{hB_{ih^{3/2}+k}: i\in \Z\}$, hit by an $x$-near geodesic from $hB_{ih^{3/2}+k}$ to $hB_{jh^{3/2}+k}$, for some $i,j\in \Z$ with $0\le i(hw)^{3/2}+k<j(hw)^{3/2}+k\le n$.
From this definition we have $H^x(n,w;h) \le \sum_{k=0}^{h^{3/2}-1}H_k$, and by Theorem~\ref{thm:near-loc-time} we have $\P[H_k > M(\log(w))^{32}n^{1/3}(hw)^{-1/2}] < Ce^{-cM}$ for each $0\le k \le h^{3/2}-1$ (here we use that $(\log(w))^{32}n^{1/3}(hw)^{-1/2} > 1$).
Then the conclusion follows from a union bound.
\end{proof}

\subsection{Inductive proof for Theorem \ref{thm:near-loc-time}}

We fix $w\in \N$. It suffices to prove for the case where $x=w^{1/2}$, so we will assume $x=w^{1/2}$ herein. Thus for simplicity of notations, we also write $H_{u,v}=H_{u,v}^x(w)$, $H(p,q)=H^x(p,q,w)$, and $H(n)=H^x(n,w)$.

We shall prove the following statements:
there exist $C,c>0$ and large $A>0$, such that 
\begin{equation}  \label{eq:hbd-e}
\E[H(n)] < A^{2/3+0.01}n^{1/3}w^{-1/2}(\log(w))^{32}+1,    
\end{equation}
and for any $M>0$
\begin{equation}  \label{eq:hbd-t}
\P[H(n) > MAn^{1/3}w^{-1/2}(\log(w))^{32}+1] < Ce^{-cM}.
\end{equation}
We prove these results by induction in $n$.
The base case where $n\le Aw^{3/2}$ is obvious, since there is always $H(n)\le nw^{-3/2}+1$, and with $n\le Aw^{3/2}$ we have $H(n) \le A^{2/3}n^{1/3}w^{-1/2}+1$.
Below we prove these results for some $n> Aw^{3/2}$, assuming they hold for all smaller $n$.

\noindent\textbf{Expectation: proof of \eqref{eq:hbd-e}.}

For any $u \preceq v \in \Z^2$, and $k\in \Z$, we let $\cI_{u,v}^k$ be the following event: there exists a $w^{1/2}$-near geodesic $\gamma$ from $u$ to $v$, with $\gamma\cap B_k\neq \emptyset$.

For $p<q$, we let $\oH(p,q)$ be the number of the following barriers $B_k$: 
\begin{enumerate}
    \item $p\le w^{3/2}k \le q$;
    \item there are some $i, j\in\Z$ and $u\in B_i, v\in B_j$, such that either $2p-q\le w^{3/2}i\le p$ and $q\le w^{3/2}j$, or $w^{3/2}i\le p$ and $q\le w^{3/2}j\le 2q-p$; 
    and $\cI_{u,v}^k$ holds.
\end{enumerate}
In words, for $\oH(p,q)$ we need the end points $u$ and $v$ satisfy that $2(2p-q) \le d(u) \le 2p$ and $2q \le d(v) \le 2(2q-p)$, whereas for $H(p,q)$ we have $2p \le d(u), d(v) \le 2q$.
This $\oH(p,q)$ is easier to work with, and we can directly bound its expectation, using translation invariance of the model.
\begin{lemma}  \label{lem:bound-oH}
There is a universal constant $C_1>0$, such that $\E[\oH(p,q)]<C_1(\log(w))^{32} ((q-p)^{1/3}w^{-1/2}+ 1)$, for any $p<q\in\Z$, $q-p < (w^{3/2})^{1.03}$.
\end{lemma}
The following estimate is a key input.
\begin{lemma}  \label{lem:multi-dis-input}
For $l\in \N$, let $\cI_l$ be the following event: there exist $u\in B_i$ for some $i\in \Z$ with $-2l<iw^{3/2}<-l$, and also $v\in \Z^2$ with $d(v)>w^{3/2}/10$, such that the event $\cI_{u,v}^0$ holds.
There is a universal constant $C_1'>0$, such that $\P[\cI_l]<C_1'(\log(w))^{32}wl^{-2/3}$, for any $w^{3/2}<l < (w^{3/2})^{1.02}$.
\end{lemma}
We leave the proof of this lemma to the next subsection, and deduce Lemma \ref{lem:bound-oH} assuming it.

\begin{proof}[Proof of Lemma \ref{lem:bound-oH}]
The general idea is to do a dyadic decomposition on the location of the end points, and apply Lemma \ref{lem:multi-dis-input}.

We start by setting up some notations.
For each $k\in\Z$, and $a<b<g$, let $\cE_{a,b,g,k}^-$ be the event where there are some $i,j\in\Z$, $a\le w^{3/2}i\le b$, $w^{3/2}j\ge g$, and $u\in B_i$, $v\in B_j$, such that $\cI_{u,v}^k$ holds.
For each $k\in\Z$, and $a>b>g$, let $\cE_{a,b,g,k}^+$ be the event where there are some $i,j\in\Z$, $a\ge w^{3/2}j\ge b$, $w^{3/2}i\le g$, and $u\in B_i$, $v\in B_j$, such that $\cI^k_{u,v}$ holds.

We take any $k\in\Z$ with $p+2w^{3/2}< w^{3/2}k \le (p+q)/2$, and any $m\in\N$ such that $2^m>w^{3/2}$, and $w^{3/2}k-2^{m+1} \le p$, $w^{3/2}k-2^m \ge 2p-q$. Then we have $w^{3/2}<2^m < (w^{3/2})^{1.03}$.
By Lemma \ref{lem:multi-dis-input}, we have
\[
\P[\cE_{w^{3/2}k-2^{m+1}, w^{3/2}k-2^m, q, k}^-] < C_1'2^{-2m/3}w(\log(w))^{32}.
\]
Now by taking a union bound over all such $m$, we conclude that 
\[
\P[\cE_{2p-q,p,q,k}^-] < 10C_1' (w^{3/2}k-p)^{-2/3}w(\log(w))^{32}.
\]
For each $k\in\Z$ with $p\le w^{3/2}k\le (p+q)/2$ we can similarly get 
\[
\P[\cE_{2q-p,q,2p-q,k}^+] < 10C_1' (q-p)^{-2/3}w(\log(w))^{32}.
\]
Note that $\cE_{2p-q,p,q,k}^-\cup \cE_{2q-p,q,2p-q,k}^+$ precisely implies that $\cI_{u,v}^k$ holds for some $u\in B_i, v\in B_j$, where $i, j\in\Z$ and either $2p-q\le w^{3/2}i\le p$ and $q\le w^{3/2}j$, or $w^{3/2}i\le p$ and $q\le w^{3/2}j\le 2q-p$.
The previous two inequalities imply that we can bound its probability by $20C_1' (w^{3/2}k-p)^{-2/3}w(\log(w))^{32}$, for any $k\in\Z$ with $p+2w^{3/2}< w^{3/2}k \le (p+q)/2$.
By symmetry, for $k\in\Z$ with $(p+q)/2\le w^{3/2}k< q-2w^{3/2}$ we can bound this probability by $20C_1' (q-w^{3/2}k)^{-2/3}w(\log(w))^{32}$.
Then by summing over all $k$ we get the conclusion.
\end{proof}

Next we shall bound $\E[H(n)]$, using the bound of $\E[\oH(p,q)]$ and the induction hypothesis.
The idea is again using some dyadic decomposition on the location of the end points.

Take some $g\in \N$ to be determined, and denote $G=2^g$. We take two collections of intervals:
\[
\Phi = \{(a2^b, (a+1)2^b): a,b, \in \Z, \; b>g, \; 0\le a2^b<(a+1)2^b \le n \},
\]
and 
\[
\Psi = \{(a2^g, (a+1)2^g): a \in \Z, \; 0\le a2^g<(a+1)2^g \le n+2^g \}.
\]
We claim that
\begin{equation}  \label{eq:cover}
H(n) \le \sum_{(p,q)\in\Phi} \oH(p,q) + 2\max_{(p,q)\in\Psi} H(p,q) .
\end{equation}

\begin{figure}[hbt!]
    \centering
\begin{tikzpicture}[line cap=round,line join=round,>=triangle 45,x=7cm,y=7cm]
\clip (-0.3,-0.3) rectangle (1.7,1.7);

\foreach \i in {0,...,75}
{
\draw [thin, blue] (\i/50-0.005, \i/50+0.005) -- (\i/50+0.005, \i/50-0.005);
}

\draw [red] plot [smooth] coordinates {(1.15,1.13) (1.15,1.03) (0.92,1.03) (0.92,0.89) (0.63,0.89) (0.63,0.6) (0.45,0.6) (0.45, 0.32) (0.29,0.32) (0.29,0.27) (0.29,0.22) (0.2,0.22) (0.2,0.16)};

\begin{scriptsize}

\draw (0.725+1/32,1.525+1/32) node[anchor=south]{$\L_{2(a_++1)2^{g}}$};
\draw (0.725,1.525) node[anchor=east]{$\L_{2(a_{1,+}+1)2^{b_{1,+}}}=\L_{2a_+2^{g}}$};
\draw (0.6,1.4) node[anchor=east]{$\L_{2(a_0+1)2^{b_0}}=\L_{2a_{1,+}2^{b_{1,+}}}$};
\draw (0.1,0.9) node[anchor=east]{$\L_{2a_02^{b_0}}$};

\draw (0.2,0.16) node[anchor=north west]{$u$};
\draw (1.14,1.14) node[anchor=north west]{$v$};

\draw (0.,0.) node[anchor=east]{$(0,0)$};
\draw (1.5,1.5) node[anchor=west]{$(n,n)$};

\draw (0.6,-0.2) node[anchor=north]{$2^{g}$};
\draw (0.67,-0.13) node[anchor=north]{$2^{b_{2,-}}$};

\draw (0.8,0.) node[anchor=north]{$2^{b_{1,-}}$};
\draw (1.2,0.4) node[anchor=north]{$2^{b_0}$};
\draw (1.5,0.7) node[anchor=north]{$2^{b_{1,+}}$};

\draw (1.57,0.77) node[anchor=north]{$2^{g}$};
\end{scriptsize}

\draw [dotted, thin] (0.725+1/32,1.525+1/32) -- (1.525+1/32,0.725+1/32);
\draw [dotted, thin] (0.725,1.525) -- (1.525,0.725);
\draw [dotted, thin] (0.6,1.4) -- (1.4,0.6);

\draw [dotted, thin] (0.1,0.9) -- (0.9,0.1);
\draw [dotted, thin] (-0.15,0.65) -- (0.65,-0.15);
\draw [dotted, thin] (-0.15-1/16,0.65-1/16) -- (0.65-1/16,-0.15-1/16);
\draw [dotted, thin] (-0.15-1/16-1/32,0.65-1/16-1/32) -- (0.65-1/16-1/32,-0.15-1/16-1/32);
\draw [fill=uuuuuu] (0,0) circle (1.0pt);
\draw [fill=uuuuuu] (1.5,1.5) circle (1.0pt);
\draw [fill=uuuuuu] (1.14,1.14) circle (1.0pt);
\draw [fill=uuuuuu] (0.2,0.16) circle (1.0pt);

\draw[{to}-{to}] (0.65-1/16-0.03,-0.15-1/16+0.03) -- (0.65-1/16-1/32-0.03,-0.15-1/16-1/32+0.03);

\draw[{to}-{to}] (0.65-1/16-0.03,-0.15-1/16+0.03) -- (0.65-0.03,-0.15+0.03);

\draw[{to}-{to}] (0.65-0.03,-0.15+0.03) -- (0.9-0.03,0.1+0.03);

\draw[{to}-{to}] (0.9-0.03,0.1+0.03) -- (1.4-0.03,0.6+0.03);
\draw[{to}-{to}] (1.525-0.03,0.725+0.03) -- (1.4-0.03,0.6+0.03);
\draw[{to}-{to}] (1.525-0.03,0.725+0.03) -- (1.525+1/32-0.03,0.725+1/32+0.03);
\end{tikzpicture}
\caption{An illustration of the proof of \eqref{eq:cover}: we take a dyadic decomposition of a path from $u$ to $v$. The blue segments illustrate the barriers.}
\label{fig:dec-H}
\end{figure}
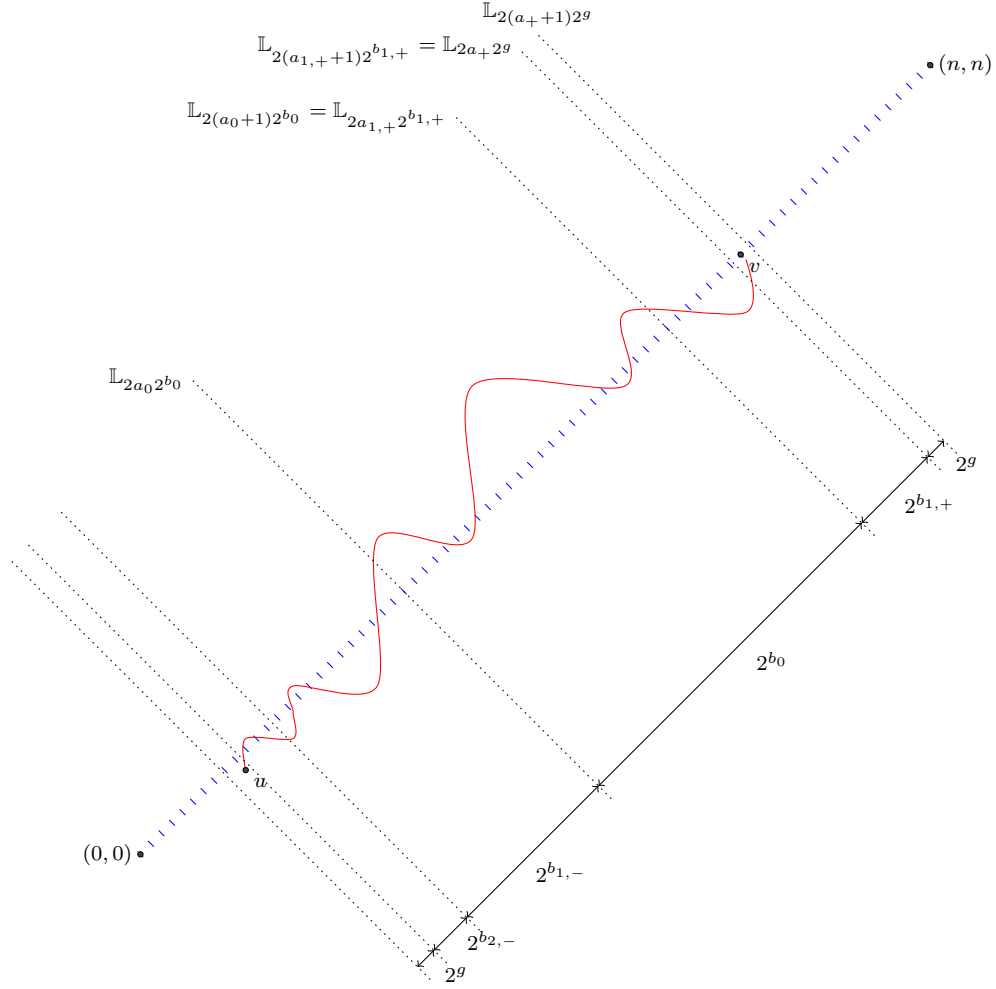

\begin{proof}[Proof of \eqref{eq:cover}]
Take any $i,j\in \Z$ such that $0\le w^{3/2}i<w^{3/2}j \le n$, and $u\in B_i$, $v\in B_j$.
We need to bound $H_{u,v}$ by the right hand side of \eqref{eq:cover}.

We first let $b_0$ be the largest integer, such that there is some $a_0\in\Z$, with $(a_0-1)2^{b_0} < w^{3/2}i\le a_02^{b_0} < (a_0+1)2^{b_0} \le w^{3/2}j$.
As $b_0$ is the largest such integer, we must also have that $w^{3/2}j<(a_0+3)2^{b_0}$.

If $b_0<g$, we can find some $a_*\in \Z$ such that $0\le a_*2^g\le w^{3/2}i<w^{3/2}j\le (a_*+2)2^g\le n+2^g$. Then we have
\[
H_{u,v} \le H(a_*2^g, (a_*+1)2^g) + H((a_*+1)2^g, (a_*+2)2^g).
\]
This is obviously bounded by the right hand side of \eqref{eq:cover}.

We next assume that $b_0\ge g$. 
We then want to choose integers $b_{0,+}>b_{1,+}> \cdots > b_{k',+}$ and $a_{0,+},a_{1,+}, \cdots, a_{k',+}$, and $a_+$, such that
\begin{equation}  \label{eq:cover-pf-u}
[(a_0+1)2^{b_0}, w^{3/2}j] \subset \bigcup_{0\le k \le k'}[a_{k,+}2^{b_{k,+}}, (a_{k,+}+1)2^{b_{k,+}}] \cup [a_+2^g, (a_++1)2^g];
\end{equation}
and for each $0\le k \le k'$ we have
\begin{equation} \label{eq:cover-pf1}
(a_{k,+}+1)2^{b_{k,+}} \le w^{3/2}j < (a_{k,+}+2)2^{b_{k,+}}.
\end{equation}
We choose these numbers inductively.
First we take $b_{0,+}=b_0$ and $a_{0,+}\in \{a_0, a_0+1\}$, such that \eqref{eq:cover-pf1} for $k=0$ holds.

Then suppose we have defined $b_{k,+}$ and $a_{k,+}$, and we assume that \eqref{eq:cover-pf1} holds.
We consider $\lfloor \log_2( w^{3/2}j-(a_{k,+}+1)2^{b_{k,+}})\rfloor$: if it is at least $g$, we denote it as $b_{k+1,+}$, and otherwise we stop here and let $k'=k$.
Then (in the first case) we would have $b_{k+1,+}<b_{k,+}$ by \eqref{eq:cover-pf1} and the choice of $b_{k+1,+}$. We take $a_{k+1,+}$ such that $(a_{k,+}+1)2^{b_{k,+}} = a_{k+1,+}2^{b_{k+1,+}}$.
By the largest property of $b_{k+1,+}$ we would have $(a_{k+1,+}+1)2^{b_{k+1,+}} \le w^{3/2}j < (a_{k+1,+}+2)2^{b_{k+1,+}}$.

Finally we take $a_+\in \Z$ such that $a_+2^g \le w^{3/2}j < (a_++1)2^g$, then we would have $a_+2^g = (a_{k',+}+1)2^{b_{k',+}}$

For each $0\le k \le k'$, we consider the number of barriers $B_l$, such that $a_{k,+}2^{b_{k,+}} \le  w^{3/2}l \le (a_{k,+}+1)2^{b_{k,+}}$ and $\cI^l_{u,v}$ holds.
This is at most $\oH(a_{k,+}2^{b_{k,+}}, (a_{k,+}+1)2^{b_{k,+}})$.
For barriers $B_l$ such that $a_+2^g \le  w^{3/2}l \le (a_++1)2^g$ and $\cI^l_{u,v}$ holds, the number is at most $H(a_+2^g, (a_++1)2^g)$.

We can similarly inductively choose $b_{0,-}>b_{1,-}> \cdots > b_{k'',-}$ and $a_{0,-},a_{1,-}, \cdots, a_{k'',-}$, and $a_-\in \Z$, such that $b_{0,-}=b_0$ and $a_{0,-}= a_0$, and
\begin{equation}  \label{eq:cover-pf-d}
[w^{3/2}i, (a_0+1)2^{b_0}] \subset \bigcup_{0\le k \le k''}[a_{k,-}2^{b_{k,-}}, (a_{k,-}+1)2^{b_{k,-}}] \cup [a_-2^g, (a_-+1)2^g],
\end{equation}
and for each $0\le k \le k''$ we have
\[
(a_{k,+}-1)2^{b_{k,+}} < w^{3/2}i \le a_{k,+}2^{b_{k,+}}.
\]
Similarly, for each $0\le k \le k''$, we consider the number of barriers in $B_l$ such that $a_{k,-}2^{b_{k,-}} \le  w^{3/2}l \le (a_{k,-}+1)2^{b_{k,-}}$ and $\cI^l_{u,v}$ holds.
This is at most $\oH(a_{k,-}2^{b_{k,-}}, (a_{k,-}+1)2^{b_{k,-}})$. For barriers $B_l$ such that $a_-2^g \le  w^{3/2}l \le (a_-+1)2^g$ and $\cI^l_{u,v}$ holds, the number is at most $H(a_-2^g, (a_-+1)2^g)$.
Also note that $(a_{k,+}2^{b_{k,+}}, (a_{k,+}+1)2^{b_{k,+}})$ for $0\le k \le k'$ and $(a_{k,-}2^{b_{k,-}}, (a_{k,-}+1)2^{b_{k,-}})$ for $1\le k\le k''$ are mutually different. 

Thus we have that $[w^{3/2}i, w^{3/2}j]$ is covered by the union of the right hand sides of \eqref{eq:cover-pf-u} and \eqref{eq:cover-pf-d}, so we conclude that
\begin{multline*}
H_{u,v} \le \sum_{k=\iota}^{k'} \oH(a_{k,+}2^{b_{k,+}}, (a_{k,+}+1)2^{b_{k,+}}) + H(a_+2^g, (a_++1)2^g) \\ + \sum_{k=0}^{k''} \oH(a_{k,-}2^{b_{k,-}}, (a_{k,-}+1)2^{b_{k,-}})  + H(a_-2^g, (a_-+1)2^g),   
\end{multline*}
where $\iota = 0$ if $a_{0,+}=a_0+1$, and $\iota=1$ if $a_{0,+}=a_0$.
Note that the right hand side is bounded by the right hand side of \eqref{eq:cover}, so the conclusion follows.
\end{proof}

Now by taking expectations of \eqref{eq:cover} and using Lemma \ref{lem:bound-oH}, we conclude that
\[
\E[H(n)]
< 10C_1(nG^{-2/3}w^{-1/2}+n/G)(\log(w))^{32}  + 2\E[\max_{(p,q)\in\Psi} H(p,q)].
\]
Suppose that $n>G$. By the induction hypothesis, for each $(p,q)\in\Psi$ and $M>0$, we have
\[
\P[H(p,q)>MAG^{1/3}w^{-1/2}(\log(w))^{32}+1] < C e^{-cM}.
\]
Note that $|\Psi|\le n/G+1$, so we have
\[
\E[\max_{(p,q)\in\Psi} H(p,q)] < c^{-1} (\log(C)+\log(n/G+1)+1)AG^{1/3}w^{-1/2}(\log(w))^{32} +1,
\]
so by taking $G=n/A$ and using that $n> Aw^{3/2}$, we get
\begin{align*}
\E[H(n)]
< & 10C_1(n^{1/3}A^{2/3}w^{-1/2}+A)(\log(w))^{32} \\ & + 2c^{-1} (\log(C)+\log(A+1)+1) A^{2/3}n^{1/3}w^{-1/2}(\log(w))^{32} + 2\\ < & 20(C_1+c^{-1}\log(C))A^{2/3}\log(A)n^{1/3}w^{-1/2}(\log(w))^{32}+2.
\end{align*}
By taking $A$ large enough (and in particular $20(C_1+c^{-1}\log(C))\log(A)<A^{0.001}$), we get the desired bound \eqref{eq:hbd-e} for $\E[H(n)]$.

\noindent\textbf{Exponential tails: proof of \eqref{eq:hbd-t}.}

Let $r=\lfloor A^{1/10} \rfloor$, and (without loss of generality and for simplicity of notations) we assume that $n/r$ is an integer. By the induction hypothesis, we have that for any $M>0$,
\begin{equation}  \label{eq:iext}
\P[H(n/r)\ge MAr^{-1/3}n^{1/3}w^{-1/2}(\log(w))^{32}+1]\leq Ce^{-cM}\,, \end{equation}
and 
\begin{equation}  \label{eq:iex}
\E[H(n/r)]\le A^{2/3+0.01}r^{-1/3}n^{1/3}w^{-1/2}(\log(w))^{32}+1.
\end{equation}
We have
\[H(n)\le \sum_{i=1}^r H_i\,,\]
where $H_i:=H((i-1)n/r,in/r)$ for $i=1,2,\ldots, t$. By translation invariance, $H_i$'s are independent with distributions being the same as $H(n/r)$ or $H(n/r-w^{3/2})$.
We then have
\begin{align}\label{e:twoparts}
  H(n)\le  \sum_{i=1}^{r}H_i \leq & \sum_{i=1}^r H_i\ind(H_i< Ar^{-1}n^{1/3}w^{-1/2}(\log(w))^{32})\nonumber\\& +\sum_{i=1}^r H_i\ind(H_i\ge Ar^{-1}n^{1/3}w^{-1/2}(\log(w))^{32})\nonumber\\
    \leq & An^{1/3}w^{-1/2}(\log(w))^{32}+\sum_{i=1}^rH_i\ind(H_i\ge Ar^{-1}n^{1/3}w^{-1/2}(\log(w))^{32})\,.
     \end{align}
By Markov inequality and \eqref{eq:iex}, we have
\begin{multline}\label{e:markov}
  \P[H_i\ge Ar^{-1}n^{1/3}w^{-1/2}(\log(w))^{32}]\leq \frac{\E[H_i]}{Ar^{-1}n^{1/3}w^{-1/2}(\log(w))^{32}}\leq\frac{\E[H(n/r)]}{Ar^{-1}n^{1/3}w^{-1/2}(\log(w))^{32}}\\ 
  \le \frac{A^{2/3+0.01}r^{-1/3}n^{1/3}w^{-1/2}(\log(w))^{32}+1}{Ar^{-1}n^{1/3}w^{-1/2}(\log(w))^{32}}
  \leq 2A^{-1/3+0.01}r^{2/3}\,.  
    \end{multline}
Let $X_i:=A^{-1}r^{1/3}n^{-1/3}w^{1/2}(\log(w))^{-32} H_i\ind (H_i\ge Ar^{-1}n^{1/3}w^{-1/2}(\log(w))^{32})$. Then using \eqref{eq:iext} and \eqref{e:markov}, for any $t>0$ we have
\[\P[X_i\ge t]\leq Ce^{-ct+c}\wedge 2A^{-1/3+0.01}r^{2/3}\,.\]
Thus, for any $0<a<c$, 
\begin{multline*}
\E[e^{aX_i}]=1+\int_0^\infty ae^{at}\P[X_i\geq t]\rd t \\
\le 1+\int_0^{c^{-1}\log(r)} 2ae^{at}A^{-1/3+0.01}r^{2/3}\rd t+\int_{c^{-1}\log(r)}^\infty ae^{at}Ce^{-ct+c}\rd t\leq 1+\frac{C'}{r^{c'}},
\end{multline*}
where $c'=1-ac^{-1}$ and $C'$ is some positive constant (depending on $C, c$ but independent of $A$). Choosing $a=c/4$ and using the inequality $1+z\leq e^z$ for all $z>0$, we have
\[\E[e^{cX_i/4}]\le 1+C'r^{-3/4}\le e^{C'r^{-3/4}}\,.\]

Finally, using \eqref{e:twoparts} and the Chernoff inequality, for $M>2$ we have
 \begin{align*}
     &\P[H(n)\ge MAn^{1/3}w^{-1/2}(\log(w))^{32}]\\ &\leq \P\left[\sum_{i=1}^r H_i\ind(H_i\ge Ar^{-1}n^{1/3}w^{-1/2})\ge
2^{-1}MAn^{1/3}w^{-1/2}(\log(w))^{32}\right]\\
&=\P\left[\sum_{i=1}^r X_i\ge 2^{-1}Mr^{1/3}\right]\\
&\le e^{-8^{-1}cMr^{1/3}}\prod_{i=1}^r\E[e^{cX_i/4}]\le e^{-8^{-1}cMr^{1/3}}e^{C'r^{1/4}}\le Ce^{-cM}\,,
 \end{align*}
where the last inequality is by choosing $A$ large enough thus $r$ large enough. For $M\le 2$ we just take $C$ large and $c$ small.
Thus we get the desired tail estimate \eqref{eq:hbd-t}.

\subsection{Barrier hitting probability}
In this subsection we prove Lemma \ref{lem:multi-dis-input}. We restate it as follows.

Let $U_{l,w}$ be the rectangle $\{(y-x, y+x): |x|<w, -2l\le y\le -l\}$.
Let $A_w=\{(-x, x): |x|<w\}$.
Let $\cE_{l,w}$ be the event where there is some $u \in U_{l,w}$ and $v\in \Z^2$ with $d(v)>l/10$, such that there is a $w^{1/2}$-near geodesic $\gamma$ from $u$ to $v$, with $\gamma\cap A_w\neq\emptyset$.
\begin{lemma}  \label{lem:multi-dis-input2}
There exists a universal constant $C>0$, such that $\P[\cE_{l,w}]<C(\log(w))^{32}wl^{-2/3}$, for any $l, w \in \N$ with $w^{3/2}<l<(w^{3/2})^{1.02}$.
\end{lemma}
This obviously implies Lemma \ref{lem:multi-dis-input}.
To prove this, we use a translation invariance argument, and consider the following random variable.
We let $U_l=U_{l,l^{2/3}}$.
For each $i\in\Z$, we let $A_w^i=A_w + (2iw, -2iw)$.
We let $\cN_{l,w}$ be the number of $i\in \Z$, such that (i) $A_w^i\subset A_{l^{2/3}}$;
(ii) there is some $u \in U_l$ and $v\in \Z^2$ with $d(v)>l/10$, such that there is a $w^{1/2}$-near geodesic $\gamma$ from $u$ to $v$ with $\gamma\cap A_w^i\neq\emptyset$.
\begin{lemma}  \label{lem:cnlw-bd}
For any $l, w \in \N$ with $w^{3/2}<l<(w^{3/2})^{1.02}$, we have $\E[\cN_{l,w}]<C(\log(w))^{32}$, where $C>0$ is a universal constant.
\end{lemma}

Assuming this we can prove Lemma \ref{lem:multi-dis-input2}.
\begin{proof}[Proof of Lemma \ref{lem:multi-dis-input2}]
For each $i\in\Z$, we let $U_{l,w}^i=U_{l,w}+ (2iw, -2iw)$, and let $\cE_{l,w}^i$ be the event where there is some $u \in U_{l,w}^i$ and $v\in \Z^2$ with $d(v)>l/10$, such that there is a $w^{1/2}$-near geodesic $\gamma$ from $u$ to $v$ with $\gamma\cap A_w^i\neq\emptyset$.
By translation invariance we have $\P[\cE_{l,w}]=\P[\cE_{l,w}^i]$ for each $i\in\Z$.
Since $U^i_{l,w}\subset U_l$ if $2|i|w+w\le l^{2/3}$, we have that $\lfloor (2w)^{-1}l^{2/3} \rfloor \P[\cE_{l,w}] \le \E[\cN_{l,w}]$.
By Lemma \ref{lem:cnlw-bd} our conclusion follows.
\end{proof}
It remains to prove Lemma \ref{lem:cnlw-bd}, and we will need the following estimate, on the number of near geodesics between a pair of end points.

For any points $u\preceq v$, any $g>0$, $m\in\N$, and $h \in \Z$ with $d(u)<h<d(v)$, we let $\MP(u,v;h,g,m)$ be the following event: there exist $m$ points $w_1, \ldots, w_m \in \L_h$, such that
\begin{enumerate}
    \item $ad(w_{i+1})-ad(w_i)>g$, for each $i=1,\ldots, m-1$,
    \item $T_{u,w_i}+T_{w_i,v}-\xi(w_i) \ge T_{u,v}-\sqrt{g}$, for each $i=1,\ldots, m$; i.e. there is a $\sqrt{g}$-near geodesic from $u$ to $v$, passing through $w_i$.
\end{enumerate}
We also take $n\in\N$, and assume that $(n^{2/3})^{0.95}<g<n^{2/3}$ and $\log(n)^{20}<m<n^{0.05}$.
We would now bound the probability of such multiple peak events.
\begin{proposition} \label{prop:multiple-peak}
For any $\psi\in(0,1)$, there are constants $c,C>0$ depending on $\psi$, such that the following is true.
Take any $h,n,m\in \N$, $g>0$, and $u=(u_1,u_2)\in \L_n$, such that $g, m, n$ satisfy the above relations, $\psi n< h < (1-\psi)n$, and $\psi<\frac{u_1}{u_2}<\psi^{-1}$.
Then we have
\[
\P[\MP(\boo,u;h,g,m)]<Ce^{-c(\log(n))^2}.
\]
\end{proposition}
The proof of this proposition is via a connection between the point-to-line passage time profiles and random walk bridges, by the RSK correspondence.
We leave the proof to Section \ref{s:multi}, and prove Lemma \ref{lem:cnlw-bd} now.
The general idea is to rank the near geodesics to apply Proposition \ref{prop:multiple-peak}, and use Theorem \ref{thm:disjoint-near-geo} as a key input.
\begin{proof}[Proof of Lemma \ref{lem:cnlw-bd}]
We denote $\ol=\lfloor l/20\rfloor$, for simplicity of notations.
We shall let $C, c>0$ denote large and small universal constants, and the values can change from line to line.

For simplicity of notations, we take $\oM=(\log(l))^{12}$ and $M=cM^{1/33}$.

Let $B_-$ be the segment whose end points are $(-\ol-Ml^{2/3}, -\ol+Ml^{2/3})$ and $(-\ol+Ml^{2/3}, -\ol-Ml^{2/3})$, and let $B_+=B_-+(2\ol, 2\ol)$.
We consider the following events.
\begin{itemize}
    \item $\cE_1$: there is a $w^{1/2}$-near geodesic $\gamma$ from $U_l$ to $\L_{2\ol}$, satisfying that $\gamma\cap A_{l^{2/3}}\neq\emptyset$ and $\gamma$ is disjoint from $B_-$ or $B_+$.
    \item $\cE_2$: there is some $u\in B_-$ and $v\in B_+$, such that $\MP(u,v;0,w,\lceil(\log(l))^{20}\rceil)$ holds.
    \item $\cE_3$: there are $\oM$ $w^{1/2}$-near geodesics from $B_-$ to $B_+$, denoted as $\gamma_1, \ldots, \gamma_{\oM}$, such that for any $1\le j<k\le \oM$, the intersection $\gamma_j\cap \gamma_k$ is either empty, or contained below $\L_0$, or contained above $\L_0$.
\end{itemize}
We now assume that $\cE_1^c \cap \cE_2^c \cap \cE_3^c$ holds, and bound $\cN_{l,w}$.
This is achieved by ranking the near geodesics, as follows.

\begin{figure}[hbt!]
    \centering
\begin{tikzpicture}[line cap=round,line join=round,>=triangle 45,x=6cm,y=6cm]
\clip (-0.1,-0.1) rectangle (1.1,1.1);

\draw (0,0.4) -- (0.4,0);
\draw (1,0.6) -- (0.6,1);

\foreach \i in {1,...,20}
{
\draw[|-|] (1 - \i/20, \i/20) -- (1.05 - \i/20, \i/20-0.05);
}

\draw [red,smooth] plot coordinates {(0.2,0.2) (0.33,0.67) (0.7,0.9) };
\draw [red,smooth] plot coordinates {(0.28,0.12) (0.63,0.37) (0.8,0.8) };
\draw [red,smooth] plot coordinates {(0.11,0.29) (0.43,0.57) (0.9,0.7) };

\draw [blue,smooth] plot coordinates {(0.2,0.2) (0.62,0.38) (0.7,0.9) };
\draw [blue,smooth] plot coordinates {(0.28,0.12) (0.67,0.33) (0.8,0.8) };
\draw [blue,smooth] plot coordinates {(0.11,0.29) (0.27,0.73) (0.9,0.7) };

\draw [darkgreen,smooth] plot coordinates {(0.2,0.2) (0.42,0.58) (0.7,0.9) };
\draw [darkgreen,smooth] plot coordinates {(0.28,0.12) (0.57,0.43) (0.8,0.8) };
\draw [darkgreen,smooth] plot coordinates {(0.11,0.29) (0.53,0.47) (0.9,0.7) };

\draw[blue, ultra thick] (0.255,0.745) -- (0.295,0.705);
\draw[red, ultra thick] (0.305,0.695) -- (0.345,0.655);
\draw[red, ultra thick] (0.405,0.595) -- (0.445,0.555);
\draw[darkgreen, ultra thick] (0.505,0.495) -- (0.545,0.455);
\draw[darkgreen, ultra thick] (0.595,0.405) -- (0.555,0.445);
\draw[red, ultra thick] (0.605,0.395) -- (0.645,0.355);
\draw[blue, ultra thick] (0.695,0.305) -- (0.655,0.345);

\begin{scriptsize}
\draw (0.,1) node[anchor=east]{$\L_0$};
\draw (0.,0.4) node[anchor=east]{$B_-$};
\draw (0.6,1) node[anchor=east]{$B_+$};
\end{scriptsize}

\draw [fill=uuuuuu] (0.2,0.2) circle (1.0pt);
\draw [fill=uuuuuu] (0.28,0.12) circle (1.0pt);
\draw [fill=uuuuuu] (0.11,0.29) circle (1.0pt);
\draw [fill=uuuuuu] (0.7,0.9) circle (1.0pt);
\draw [fill=uuuuuu] (0.8,0.8) circle (1.0pt);
\draw [fill=uuuuuu] (0.9,0.7) circle (1.0pt);

\end{tikzpicture}
\caption{An illustration of the proof of Lemma \ref{lem:cnlw-bd}. The line $\L_0$ is divided into $A_w^i$ for $i\in \Z$.
The red, blue, green curves represent paths of rank-1, rank-2, rank-3 respectively, for the corresponding end points. The red, blue, green segments represent segments of rank-1, rank-2, rank-3 respectively.
The rank of a segment is the smallest rank of all paths passing through it.
}
\label{fig:rank}
\end{figure}
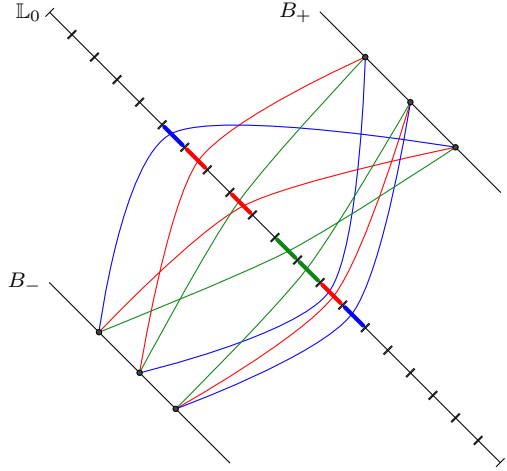

For each $u \in B_-$ and $v\in B_+$, and $i\in \Z$ with $A_w^i \subset A_{l^{2/3}}$, we let \[T_{u,v,i}=\max_{\gamma \text{ from $u$ to $v$, }, \gamma\cap A_w^i \neq \emptyset} T(\gamma).\]
In other words, $T_{u,v,i}$ is the maximum passage time from $u$ to $v$, for paths passing through $A_w^i$.
We also let $J[u,v]=\{i: T_{u,v,i}\ge T_{u,v}-w^{1/2}\}$, the set of indices of segments that intersect with some $w^{1/2}$-near geodesic from $u$ to $v$.
By $\cE_2^c$ we have $|J[u,v]|<2(\log(l))^{20}$.
We let $\iota[u,v,k]\in J[u,v]$ denote the index of the `$k$-th largest segment' for $1\leq k \leq |J[u,v]|$. In other words, we must have
\[
k = \big|\{i\in J[u,v]: T_{u,v,i}\ge T_{u,v,\iota[u,v,k]}\}\big|.
\]
The up-right path from $u$ to $v$ passing through $A^{\iota[u,v,k]}_w$ with the maximum weight is called the rank-$k$ path from $u$ to $v$.
We note that the rank-$1$ path from $u$ to $v$ is just the geodesic $\Gamma_{u,v}$, unless $\Gamma_{u,v}$ is disjoint from $\bigcup_{A^i_w\subset A_{l^{2/3}}} A^i_w$.

For each $i\in \Z$ with $A_w^i \subset A_{l^{2/3}}$, if $i\notin \bigcup_{u\in B_-, v\in B_+}J[u,v]$, we let $\kappa[i]=\infty$; otherwise, we let $\kappa[i]$ be the smallest number, such that $i=\iota[u,v,\kappa[i]]$ for some $u\in B_-$, $v\in B_+$. In other words, $\kappa[i]$ is the smallest rank of path passing through $A_w^i$ (from some $u\in B_-$ to some $v\in B_+$); and we  call $\kappa[i]$ the rank of $A_w^i$.
By $\cE_2^c$ we have that $\kappa[i] < 2(\log(l))^{20}$ if $\kappa[i]<\infty$.

Suppose that there is some $k_* \in \N$ such that $|\{i\in \Z: A_w^i \subset A_{l^{2/3}}, \kappa[i]=k_*\}| \ge \oM$.
We can then find distinct $i_1,\ldots, i_{\oM} \in \Z$, such that for each $1\le j \le \oM$ the following is true: (i) $A_w^{i_j}\subset A_{l^{2/3}}$; (ii) there is some $u_j\in B_-$ and $v_j \in B_+$, such that $i_j=\iota[u_j,v_j,k_*]$; (iii)
\begin{equation}  \label{eq:ijnotin}
i_j \not\in \{ \iota[u,v,k]: u\in B_-, v\in B_+, 1\le k<k_* \}.
\end{equation}
We let $\gamma_j$ be the path from $u_j$ to $v_j$ with $\gamma_j\cap A_w^{i_j}\neq \emptyset$ and $T(\gamma_j)=T_{u_j,v_j,i_j}\ge T_{u_j,v_j}-w^{1/2}$.
\begin{claim}
For any $1\le j<j'\le \oM$, the intersection $\gamma_j\cap \gamma_{j'}$ is either empty, or contained below $\L_0$, or contained above $\L_0$.
\end{claim}
With this claim we get a contradiction with $\cE_3^c$, thus we conclude that $|\{i\in \Z: A_w^i \subset A_{l^{2/3}}, \kappa[i]=k\}| < \oM$ for any $k\in \N$.
\begin{proof}[Proof of the claim]
We argue by contradiction, and assume that this is not true for some $1\le j<j'\le \oM$.
Since $\gamma_j\cap \L_0 \subset A_w^{i_j}$ and $\gamma_{j'}\cap \L_0 \subset A_w^{i_{j'}}$, we can find some $u_*, v_* \in \gamma_j \cap \gamma_{j'}$, such that $u_*$ is below $\L_0$ and $v_*$ is above $\L_0$.
Let $\gamma_*$ be the part of $\gamma_j$ between $u', v'$, and let $\gamma_*'$ be the part of $\gamma_{j'}$ between $u', v'$.
Without loss of generality we assume that $T(\gamma_*)<T(\gamma_*')$. 
Then we consider the path $(\gamma_j \setminus \gamma_*)\cup \gamma_*'$: it is from $u_j$ to $v_j$, passing $\L_0$ through $A_w^{i_{j'}}$, and its weight is $>T(\gamma_j)=T_{u_j,v_j,i_j}$.
Then we have that $T_{u_j,v_j,i_{j'}}>T_{u_j,v_j,i_{j}}$, so $i_{j'} \in \{ \iota[u_j,v_j,k]: 1\le k<k_* \}$, which contradicts equation~\eqref{eq:ijnotin}. 
\end{proof}

Now we conclude that $\cN_{l,w}\le \bigcup_{u\in B_-, v\in B_+}|J[u,v]| < 2\oM(\log(l))^{20}$, under the event $\cE_1^c \cap \cE_2^c \cap \cE_3^c$.

By Proposition \ref{prop:multiple-peak} we have $\P[\cE_2]<Ce^{-c(\log(l))^2}$. Below we will prove that $\P[\cE_1]<Cl^{8/3}e^{-cM^3}$ and $\P[\cE_3]<Ce^{-cM^3}$.
Then we conclude that $\P[\cN_{l,w}\ge 2\oM(\log(l))^{20}] < Ce^{-c(\log(l))^2} + Cl^{8/3}e^{-c\oM^{1/11}}$.
Thus we get the conclusion, using that $\oM=(\log(l))^{12}$, and the fact that there is always $\cN_{l,w} < w^{-1}l^{2/3}$.
\end{proof}
\begin{proof}[Bound for $\cE_1$]
We first bound the probability that there is a $w^{1/2}$-near geodesic $\gamma$ from $U_l$ to $\L_{2\ol}$, satisfying that $\gamma\cap A_{l^{2/3}}\neq\emptyset$ and $\gamma$ is disjoint from $B_+$.

We take any $u\in U_l$ and $v\in\L_{2\ol}\setminus B_+$, and consider the probability that there is a $w^{1/2}$-near geodesic from $u$ to $v$ that intersects $A_{l^{2/3}}$.
When $|ad(v)|>2\ol+2l^{2/3}$, there is no up-right path from $A_{l^{2/3}}$ to $v$.
We have that, whenever $|ad(v)|\le 2\ol+2l^{2/3}$, this probability is at most $Ce^{-cM^3}$.
This is ensured by Proposition \ref{prop:trans-fluc} and Theorem \ref{t:onepoint} (note that the slope conditions are ensured by $|ad(v)|\le 2\ol+2l^{2/3}$).
Then by taking a union bound we get $Cl^{8/3}e^{-cM^3}$.

We next bound the probability that there is a $w^{1/2}$-near geodesic $\gamma$ from $U_l$ to $\L_{2\ol}$, satisfying that $\gamma\cap A_{l^{2/3}}\neq\emptyset$ and $\gamma$ is disjoint from $B_-$.

Now we take any $u\in U_l$ and $v\in A_{l^{2/3}}$, and consider the probability that there is a $w^{1/2}$-near geodesic from $u$ to $v$ that is disjoint from $B_-$.
By Proposition \ref{prop:trans-fluc} and Theorem \ref{t:onepoint}, this probability is at most $Ce^{-cM^3}$.
Then by a union bound over all such $u$ and $v$ we get $Cl^{7/3}e^{-cM^3}$.
\end{proof}

\begin{proof}[Bound for $\cE_3$]
Let $N=\lfloor \sqrt{\oM} \rfloor$, and let $B_0$ be the segment whose end points are $(-2Ml^{2/3}, 2Ml^{2/3})$ and $(2Ml^{2/3}, -2Ml^{2/3})$.
We consider the following events.
\begin{enumerate}
    \item $\cE_{3,1}$: there exists a $w^{1/2}$-near geodesic $\gamma$ from $B_-$ to $B_+$, and $\gamma$ is disjoint from $B_0$.
    \item $\cE_{3,2}$: there exist $N$ mutually disjoint $w^{1/2}$-near geodesics from $B_-$ to $B_0$.
    \item $\cE_{3,2}$: there exist $N$ mutually disjoint $w^{1/2}$-near geodesics from $B_0$ to $B_+$.
\end{enumerate}
By splitting $B_-$ and $B_+$ into segments of length $\ol^{2/3}$, and applying Theorem \ref{t:supinf} and Proposition \ref{prop:trans-fluc} to end pair of them, we have $\P[\cE_{3,1}]\le Ce^{-cM^3}$.
By Theorem  \ref{thm:disjoint-near-geo} (note that $M$ is much smaller than $N^{1/16}$) we have that $\P[\cE_{3,2}], \P[\cE_{3,3}]\le e^{-cN^{1/4}}$.
We next show that $\cE_3\setminus \cE_{3,1}\subset \cE_{3,2}\cup\cE_{3,3}$.

Assuming $\cE_3\setminus \cE_{3,1}$, we can find $w^{1/2}$-near geodesics $\gamma_1, \ldots, \gamma_{\oM}$ from $B_-$ to $B_+$, and each $\gamma_k$ intersects $B_0$; and for any $1\le j<k\le \oM$, the intersection $\gamma_j\cap \gamma_k$ is either empty, or contained below $\L_0$, or contained above $\L_0$.
Let $(i_k,-i_k) = \gamma_k\cap \L_0$ for each $1\le k \le \oM$, we then assume that $i_1<\cdots <i_{\oM}$.

For each $1\le j \le \oM$, let $\delta_j$ be the maximum number, such that there are $1\le k_1 <\cdots < k_{\delta_j} = j$, with $\gamma_{k_1}, \ldots, \gamma_{k_{\delta_j}}$ being mutually disjoint below $\L_0$.
If there is some $\delta_j \ge N$, the event $\cE_{3,2}$ holds.
Otherwise, since $\oM\ge (N-1)^2+1$, we can find some $j_1\le \cdots \le j_N$ such that $\delta_{j_1}=\cdots = \delta_{j_N}$.
Then we must have that $\gamma_{j_1}, \ldots, \gamma_{j_N}$ intersect with each other below $\L_0$.
This implies that $\gamma_{j_1}, \ldots, \gamma_{j_N}$ are mutually disjoint above $\L_0$, and $\cE_{3,3}$ holds.

Finally we conclude that $\P[\cE_3]\le \P[\cE_{3,1}]+\P[\cE_{3,2}]+\P[\cE_{3,3}]\le Ce^{-cM^3}$.
\end{proof}

\section{Proof of the main theorem} \label{s:main}
Throughout we assume that $0<\e<1/2$.
Let $Z^\varepsilon_{u,v}:=T_{u,v}^\e-T_{u,v}$ denote the difference in the weights of the geodesic $\Gamma^\e_{u,v}$ and $\Gamma_{u,v}$, and let $Z_{n}^\e:=Z^\e_{\boo,\bn}=T^\e_n-T_n$.

\begin{theorem}\label{t:induc}
For all $k\in \N$, $k\ge 200$, there exists some $\e_k>0$ such that for all $0<\e<\e_k$, and $n\le n_k$ with $n_k=\lfloor\e^{-k/200}\rfloor$,
\[\P\left[Z_n^\e>\varepsilon^{1/3} n_k^{1/3}(\log(n_k))^{50k}\right]\leq e^{-c(\log(k))^{-1}(\log(\e^{-1}))^2}\,,\]
for some universal constant $c>0$.
\end{theorem}

\begin{figure}[h]
\centering
\includegraphics[width=3.5in]{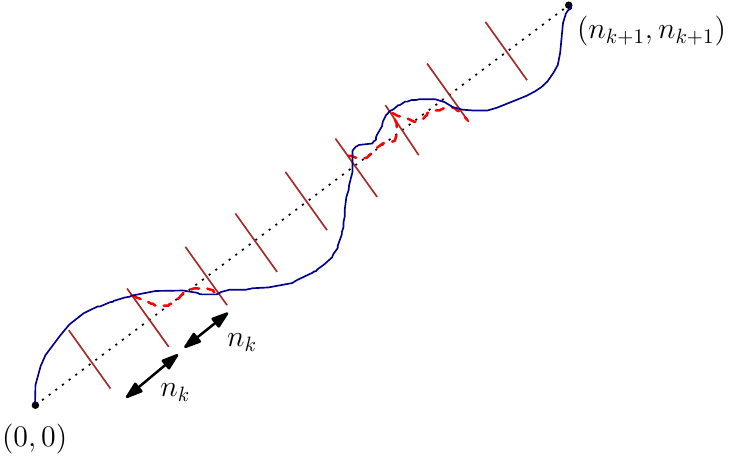}
\caption{This is an illustration of the proof of Theorem \ref{t:induc}. The total length $n_{k+1}$ is divided into $L$ blocks of length $n_k$ each. The brown segments along the anti-diagonal represent the $L$ barriers. The number of barriers hit by the original geodesic is of the order of $L^{1/3}$. If the geodesic does not hit the barrier, it is very unlikely it will hit the diagonal in between. If the geodesic hits a barrier, there may be an alternate path in that respective block that passes through the diagonal in the reinforced environment; however the maximum improvement in such a block is controlled by the induction hypothesis.}
\label{F.1}
\end{figure}

\begin{proof}
The general strategy of this proof follows the sketch given at the end of Section \ref{sec:psk}.
We leave $c$ and $\e_k$ to be determined, and let $C$ denote a large constant throughout the proof.
Without loss of generality and for the simplicity of notation, we assume that $\e^{-1/200}$ is an integer, and denote $L:=\varepsilon^{-1/200}$.

We prove by induction. In the base case where $k=200$, we have $n_k= \e^{-1}$. Since by our coupling, $Z_{n_k}^\e$ is at most the sum of $n_k$ independent random variables, each distributed as $\rho\xi$ with $\rho$ a Bernoulli($\e$) random variable and $\xi$ an independent Exp($1-\e$) random variable, the bound above follows from the usual Chernoff bound for exponential random variables.

Now we take $k\ge 200$.
Denote $c_k=c(\log(k))^{-1}$. By the induction hypothesis, we assume that for $\e>0$ small enough (depending on $k$) and all $n\le n_k$,
\[\P\left[ Z^\varepsilon_{n}>\e^{1/3} n_k^{1/3}(\log(n_k))^{50k}\right]\leq e^{-c_k(\log(\e^{-1}))^2}.\]
We now prove this for $k$ replaced by $k+1$, and all $n\le n_{k+1}$.

Note that there is nothing to prove if $n\le n_k$, so we assume $n\ge n_k$. We also assume that $n/L$ is an integer for the simplicity of notation.
Recall that we call any up-right path $\gamma$ from $u$ to $v$ an {$x$-near geodesic}, if $T(\gamma)\geq T_{u,v}-x$. 
For each $i=0,1,2,\ldots, L-1$, we call $\{u \in \Z^2: 2in/L\le d(u) < 2(i+1)n/L\}$ the $i$-th block.

As indicated in Section \ref{sec:psk}, we will use induction hypothesis to show that (a proxy of) $\Gamma^*_{\boo,\bn}$ is an $x$-near geodesic for $x$ much smaller than $n_k^{1/3}$, with high probability.
Using this, we can apply Theorem \ref{thm:near-geo-loc} to show that $\Gamma^*_{\boo,\bn}$ intersects the diagonal in roughly $L^{1/3}$ blocks, and then bound $Z_n^\e$ using the induction hypothesis for each block.

We let $\cE$ be the following event.
For any $0\le i < L$ and $u\preceq v$ in the $i$-th block and also on the diagonal, there is $Z^\e_{u,v}\le  \e^{1/3} n_k^{1/3}(\log(n_k))^{50k}$.
Then by the induction hypothesis and a union bound, we have $\P[\cE]\ge 1-Ln_k^2e^{-c_k(\log(\e^{-1}))^2}$.

We now define the proxy of $\Gamma_{\boo,\bn}^\e$. We let $\gamma_*$ be the path constructed in the following way.
For each $i$, if $\Gamma_{\boo,\bn}^\e$ does not intersect the diagonal in the $i$-th block, we let $\gamma_*$ be the same as $\Gamma_{\boo,\bn}^\e$ in that block;
otherwise, let $u[i]$ and $v[i]$ be the first and last points in the intersection of $\Gamma_{\boo,\bn}^\e$ with the diagonal in the $i$-th block, and we replace the part of $\Gamma_{\boo,\bn}^\e$ between $u[i]$ and $v[i]$ by $\Gamma_{u[i],v[i]}$.
Under $\cE$ we would have
\[
T_n - T(\gamma_*) \le T^\e_{n} - T(\gamma_*) \le L \varepsilon^{1/3} n_k^{1/3}(\log(n_k))^{50k}\le \varepsilon^{1/4}n_{k}^{1/3},
\]
for sufficiently small $\e$ (depending on $k$).  Thus $\gamma_*$ is an $\varepsilon^{1/4}n_k^{1/3}$-near geodesic (assuming $\cE$). 

Recall the barriers $(\log(n))^5 B_i$ from Section \ref{s:loc} with $w=(n/L)^{2/3}$, which are segments of length $2(n/L)^{2/3}(\log(n))^5$, perpendicular to and bisected by the diagonal, and passing through $(in/L, in/L)$.
By Theorem \ref{thm:near-geo-loc}, and recall the notations there, we have $\P[\cE']\ge 1-Ce^{-c(\log(n))^3}$, for $\cE'$ being the event where $H^{\varepsilon^{1/4}n_k^{1/3} }(n,(n/L)^{2/3};(\log(n))^5) \le (\log(n))^{40}L^{1/3}$.
(Observe that as $k\ge 200$, the two conditions $n^{0.99}<n/L$ and $\varepsilon^{1/4}n_k^{1/3} \le (n/L)^{1/3}$ are satisfied.)
Assuming $\cE\cap \cE'$, we have
\begin{equation}  \label{eq:nbab}
|\{0\le i \le L-2: (\log(n_k))^5 B_i \cap \gamma_* \neq \emptyset \}| \le (\log(n_k))^{40}L^{1/3}.
\end{equation}

For each $0\le i \le L-2$, we let $\cE^{(i)}$ be the event where any $\e^{1/4}n_k^{1/3}$-near geodesic from $\L_{2in/L}\setminus (\log(n))^5 B_i$ to $\bn$ does not hit the diagonal in the $i$-th block.
By Lemma \ref{lem:cars} below we have $\P[\cE^{(i)}]\ge 1-e^{-c(\log(n_k))^3}$.
Assume $\cE\cap \cE^{(i)}$, and suppose that $\gamma_*$ does not hit the $i$-th barrier $(\log(n))^5 B_i$, then it does not hit the diagonal in the $i$-th block.
Note that $\Gamma_{\boo,\bn}^\e$ hits the diagonal in the $i$-th block if and only if $\gamma_*$ hits the diagonal in the $i$-th block.
Thus assuming $\cE\cap \cE^{(i)}$, the contribution to $Z^\e_{n}$ from the $i$-th block would be zero if $\gamma_*$ does not hit the $i$-th barrier.
More precisely, we let $X_i=\sup_{u\in (\log(n_k))^5 B_i,v\in (\log(n_k))^5 B_{i+1}} Z_{u,v}^\e$ for each $0\le i \le L-1$.
Then assuming $\cE\cap\bigcap_{i=0}^{i-2}\cE^{(i)}$ we have
\[
Z_n^\e\leq \sum_{0\le i \le L-2: (\log(n_k))^5 B_i \cap \gamma_* \neq \emptyset} X_i + X_{L-1}.
\]
Then by \eqref{eq:nbab}, under the event $\cE\cap\cE'\cap\bigcap_{i=0}^{i-2}\cE^{(i)}$ we have
\begin{multline*}
Z_n^\e \le  ((\log(n_k))^{40}L^{1/3}+1)\max_{0\le i \le L-1}X_i  \le ((\log(n_k))^{40}L^{1/3}+1)\e^{1/3} n_k^{1/3}(\log(n_k))^{50k} \\ \le \e^{1/3} n_{k+1}^{1/3}(\log(n_{k+1}))^{50(k+1)},
\end{multline*}
where the second inequality is due to that, under $\cE$, we have $X_i\le \e^{1/3} n_k^{1/3}(\log(n_k))^{50k}$ for each $i$.

It remains to lower bound $\P\left[\cE\cap \cE'\cap \bigcap_{i=0}^{L-2}\cE^{(i)}\right]$.
From $\P[\cE]\ge 1-Ln_k^2e^{-c_k(\log(\e^{-1}))^2}$, $\P[\cE']\ge 1-Ce^{-c(\log(n))^3}$, and $\P[\cE^{(i)}]\ge 1-e^{-c(\log(n_k))^3}$ for each $i$, we have $\P\left[\cE\cap \cE'\cap \bigcap_{i=0}^{L-2}\cE^{(i)}\right]\ge 1-2Ln_k^2e^{-c_k(\log(\e^{-1}))^2}$.
Since $c_{k+1}<c_k$, by choosing $\e$ sufficiently small depending on $k$, we have $1-2Ln_k^2e^{-c_k(\log(\e^{-1}))^2} \ge 1-e^{-c_{k+1}(\log(\e^{-1}))^2}$. Thus we get the desired bound for $k+1$.
\end{proof}

We need to prove the following estimate on the local fluctuation of near geodesics.
\begin{lemma}  \label{lem:cars}
For $m, n\in \N$, $2m\le n$, let $\cE_{m,n}$ be the following event: there exists some $a, b\in \Z$, with $0\le b\le m$, $|a|\ge (\log(n))^5m^{2/3}$, and $|a|\le b$, and $T_{(a,-a),(b,b)}+T_{(b,b),\bn} \ge T_{(a,-a),\bn} - m^{1/3}$.
Then there exists universal constant $c>0$, such that $\P[\cE_{m,n}]\le e^{-c(\log(n))^3}$, when $n, m$ are large enough.
\end{lemma}
\begin{proof}
The proof is by a union bound over all choices $a, b$, using Theorem \ref{t:onepoint} and Lemma \ref{lem:lpp-devi}.
Fix $a, b\in \Z$, with $0\le b\le m$, and $(\log(n))^5m^{2/3} \le a \le b$. The case where $a<0$ follows by symmetry.
Let $a'=\lfloor a(1-b/n) \rfloor$.
Since $T_{(a,-a),\bn} \ge T_{(a,-a),(b+a'-1,b-a')} + T_{(b+a',b-a'),\bn} $,
if $T_{(a,-a),(b,b)}+T_{(b,b),\bn} \ge T_{(a,-a),\bn} - m^{1/3}$, one of the following happens:
\begin{enumerate}
    \item $T_{(b,b),\bn} - T_{(b+a',b-a'),\bn} > D_{(n-b,n-b)}-D_{(n-b-a',n-b+a')}+(\log(n))^7\sqrt{a'}$,
    \item $T_{(a,-a),(b,b)} > D_{(b-a,b+a)} + (\log(n))^3(b+|a|)^{1/2}(b-|a|+1)^{-1/6}$,
    \item $T_{(a,-a),(b+a'-1,b-a')} < D_{(b+a'-a,b-a'+a)} - (\log(n))^3m^{1/3}$.
\end{enumerate}
Indeed, by elementary computation we have
\[
D_{(n-b,n-b)}-D_{(n-b-a',n-b+a')} + D_{(b-a,b+a)} - D_{(b+a'-a,b-a'+a)} < -c{a'}^2/m,
\]
where $c>0$ is a small constant.
Then we have
\begin{multline*}
(D_{(n-b,n-b)}-D_{(n-b-a',n-b+a')}+(\log(n))^7\sqrt{a'}) + (D_{(b-a,b+a)} + (\log(n))^3(b+|a|)^{1/2}(b-|a|+1)^{-1/6})
\\
- (D_{(b+a'-a,b-a'+a)} - (\log(n))^3m^{1/3}) \\
< -c{a'}^2/m + (\log(n))^7\sqrt{a'} + (\log(n))^3(b+|a|)^{1/2}(b-|a|+1)^{-1/6} + (\log(n))^3m^{1/3} < -m^{1/3},
\end{multline*}
where the last inequality uses that $|a|\ge (\log(n))^5m^{2/3}$. 
Thus if none of the three events holds, we must have
\[
T_{(b,b),\bn} - T_{(b+a',b-a'),\bn} + T_{(a,-a),(b,b)} - T_{(a,-a),(b+a'-1,b-a')} < -m^{1/3},
\]
which contradicts with $T_{(a,-a),(b,b)}+T_{(b,b),\bn} \ge T_{(a,-a),\bn} - m^{1/3}$.

We finally bound the probabilities of the three events.
We apply Lemma \ref{lem:lpp-devi} for the first event, Theorem \ref{t:onepoint}(iv) for the second event, and Theorem \ref{t:onepoint}(ii) for the third event.
For these, we note that $a'\le b(1-b/n) \le (n-b)/2$, and $|a-a'|+1 \le ab/n+2 \le b/2+2$, so the slopes $\frac{n-b-a'}{n-b+a'}$ and $\frac{b+a'-a+1}{b-a'+a}$ are in $[1/3, 3]$ when $a, b$ are large enough.
Then the probability of each of the three events is bounded by $e^{-c(\log(n))^3}$ for some constant $c>0$, when $n, m$ are large enough.
Thus we conclude that $\P[T_{(a,-a),(b,b)}+T_{(b,b),\bn} \ge T_{(a,-a),\bn} - m^{1/3}] < 3e^{-c(\log(n))^3}$.
Note that there are no more than $n^2$ choices of $a, b$, so by taking a union bound over all $a, b$ the conclusion follows.
\end{proof}
 
It is easy to get the following improvement to Theorem \ref{t:induc}. 

\begin{corollary}\label{c:induc}
For all $k\in \N$, $k\ge 200$, there exists some $\e_k>0$ such that for all $0<\e<\e_k$, with $n_k=\lfloor\e^{-k/200}\rfloor$,
\[\P\left[\sup_{u\in \L_0, v\in \L_{2n_k}} Z^\varepsilon_{u,v}>\varepsilon n_k^{1/3}(\log(n_k))^{50k}\right]\leq e^{-c(\log(k))^{-1}(\log(\e^{-1}))^2}\,,\]
for some universal constant $c>0$.
 \end{corollary}
 
\begin{proof}This follows from Theorem \ref{t:induc} together with a union bound on the first and last times the path hits the diagonal.
\end{proof}
We next deduce our main result Theorem \ref{t:main} using the above results.
\begin{lemma}\label{lem:exp} Given any $k\in \N$, $k\ge 200$, there exists $\varepsilon_k>0$ such that for all $\varepsilon<\varepsilon_k$, and all $\l\in \N$, with $n_{k,\l}=\l\lfloor \e^{-k/200} \rfloor$, we have
\[\E[Z_{n_{k,\l}}^\varepsilon]\leq \varepsilon^{k/600}n_{k,\l}\,.\]
\end{lemma}

\begin{proof}
For the simplicity of notation, we again assume that $\e^{-1/200}$ is an integer. We let $c>0$ denote a small universal constant, and its value can change from line to line.

We first prove this theorem for $\ell = 1$. Note that
\[\E[Z^\e_{n_k}\don[Z^\e_{n_k}>n_k]]=n_k\P[Z^\e_{n_k}>n_k] + \int_{0}^\infty \P[Z^\e_{n_k}>n_k+t]\rd t\le e^{-cn_k},\]
since $Z_{n_k}^\e$ is at most the sum of $n_k$ independent random variables, each distributed as $\rho\xi$ with $\rho$ a Bernoulli($\e$) random variable and $\xi$ an independent Exp($1-\e$) random variable. And
\[
\E[Z_{n_k}^\e \don[Z_{n_k}^\e\le n_k^{1/2}]] \le n_k^{1/2}, \quad
\E[Z_{n_k}^\e \don[n_k^{1/2}\le Z_{n_k}^\e\le n_k]] \le n_k \P[Z_{n_k}^\e\ge n_k^{1/2}].
\]
Thus, we get
\[\E[Z_{n_k}^\e]\le n_k^{1/2}+n_k\P[Z_{n_k}^\e\ge n_k^{1/2}] + e^{-cn_k}.\]
Using Theorem \ref{t:induc}, we get
\[\P[Z_{n_k}^\e\ge n_k^{1/2}]\le e^{-c(\log(k))^{-1}(\log(\e^{-1}))^2}\le n_k^{-1}\]
for all $\e$ small enough depending on $k$.
Hence
\[\E[Z^\e_{n_k}]\le n_k^{1/2}+1+e^{-cn_k}\le \e^{k/600}n_k\,.\]
We then consider general $\ell$. Let $Z[i]=\sup_{u\in \L_{2in_k},v\in \L_{2(i+1)n_k}} Z^\varepsilon_{u,v}$ for each $i\in \Z_{\ge 0}$. Then we have
$Z^\e_{n_{k,\l}}\le \sum_{i=0}^{\l-1} Z[i]$.
For each $i\in \Z_{\ge 0}$, using the same arguments as above, and using Corollary \ref{c:induc} instead of Theorem \ref{t:induc}, we get $\E[Z[i]]\le\e^{k/600} n_k$.
Thus the conclusion follows.
\end{proof}

Now we are ready to prove Theorem \ref{t:main}.

\begin{proof}[Proof of Theorem \ref{t:main}]Fix any $k\in \N$, $k\ge 200$. From Lemma \ref{lem:exp}, taking limit as $\l\to \infty$, and using 
 $\lim_{n\to\infty} \E[n^{-1}T_n]=4$, we have \[\limsup_{\l\to \infty}\E[n_{k,\l}^{-1}T_{n_{k,\l}}^\varepsilon]-4\leq \varepsilon^{k/600}\,,\]
 for all $\e<\e_k$ for some $\e_k>0$.
Since by Subadditive Ergodic Theorem, $\lim_{n\to\infty} \E[n^{-1}T_n^\e]$ exists, we get from above
\[\Xi(\e)=\lim_{n\to \infty}\frac{1}{n}\E[T_n^\e]-4\le \e^{k/600}\,.\]
This proves the theorem, since $k$ is arbitrarily taken.
\end{proof}

\section{Tail of multiple near-optimums}   \label{s:multi}
In this section we prove Proposition \ref{prop:multiple-peak}.
We use a Gibbs property of the point-to-line profile in Exponential LPP, which we state now.
\begin{definition}
For each $k\in\N$ and $u\preceq v$, let $L^k_{u,v}=\max_{\gamma_1,\ldots, \gamma_k} \sum_{i=1}^k T(\gamma_i)$, where the maximum is over all (possibly empty) up-right paths $\gamma_1,\ldots, \gamma_k$, such that they are mutually disjoint, and each is contained in the set $\{w\in\Z^2: u\preceq w \preceq v\}$.
Denote $W^k_{u,v}=L^k_{u,v}-L^{k-1}_{u,v}$ (where we take $L^0_{u,v}=0$).
For any $\boo\preceq v$ we also write $W^k_v=W^k_{\boo,v}$.
\end{definition}
We list some properties of these maximum disjoint weights.
\begin{enumerate}
    \item $T_{u,v}=W^1_{u,v}$.
    \item $W^k_{u,v}\le W^{k}_{u,v+(1,0)}$ and $W^{k+1}_{u,v+(1,0)}\le W^k_{u,v}$; similarly $W^k_{u,v}\le W^{k}_{u,v+(0,1)}$ and $W^{k+1}_{u,v+(0,1)}\le W^k_{u,v}$.
    \item For $u=(u_1,u_2)\preceq v=(v_1,v_2)$ and $k>1+(v_1-u_1)\wedge (v_2-u_2)$, we must have $W^k_{u,v}=0$.
\end{enumerate}
For each $n\in\N$ we let $I_n:=\{(a,b,k):a,b\in\Z_{\ge 0}, a+b=n, k\in\llbracket 1, 1+a\wedge b\rrbracket \}$.
We now give the explicit distribution of $\{W_{(a,b)}^k\}_{(a,b,k)\in I_n\cup I_{n+1}}$.
\begin{theorem}  \label{thm:gibbs-prob}
For any non-negative $\{p_{(a,b)}^k\}_{(a,b,k)\in I_n \cup I_{n+1}}$, if it satisfies the following \emph{interlacing condition}: $p_{(a,b)}^k\le p_{(a+1,b)}^k, p_{(a,b+1)}^k$ and $p_{(a,b)}^k\ge p_{(a+1,b)}^{k+1}, p_{(a,b+1)}^{k+1}$ for any $(a,b,k)\in I_n$, 
we have
\begin{multline*}
\P\left[ W_{(a,b)}^k \in [p_{(a,b)}^k, p_{(a,b)}^k+\rd p_{(a,b)}^k), \; \forall (a,b,k)\in I_n\cup I_{n+1} \right] \\= \frac{1}{Z}\exp\left(-\sum_{(a,b,k)\in I_{n+1}} p_{(a,b)}^k + \sum_{(a,b,k)\in I_{n}} p_{(a,b)}^k\right) \prod_{(a,b,k)\in I_n\cup I_{n+1}} \rd p^k_{(a,b)},
\end{multline*}
where $Z$ is the partition function.
If the interlacing condition is not satisfied, the left hand side equals zero.
\end{theorem}
We let $S:\Z_{\ge 0}\to \R$ be a random walk, where each $S(i)-S(i-1) \sim \LAP(0,2)$, i.e., the difference of two independent $\Exp(1/2)$ random variables.
From Theorem \ref{thm:gibbs-prob} we immediately get the following Gibbs property of the profile $a\mapsto T_{\boo,(a,n-a)}$.
\begin{lemma}  \label{lem:T-Gibbs}
Take any $l<r\in \llbracket 0, n\rrbracket$. Conditioned on $\{W^2_{(a,n+1-a)}\}_{l\le a \le r+1}$ and $T_{\boo,(l,n-l)}, T_{\boo,(r,n-r)}$, the process $\{T_{\boo,(l+i,n-l-i)}-T_{\boo,(l,n-l)}\}_{i=0}^{r-l}$ has the same distribution as the random walk $\{S(i)\}_{i=0}^{r-l}$, conditioned on that $S(r-l)=T_{\boo,(r,n-r)}-T_{\boo,(l,n-l)}$, and $S(i)> W^2_{(l+i,n+1-l-i)}\vee W^2_{(l+i+1,n-l-i)} - T_{\boo,(l,n-l)}$ for each $0\le i \le l-r$.
\end{lemma}
This result can be viewed as \cite[Corollary 4.8]{O03} or \cite[Theorem 5.2]{DNV}, in slightly different settings.
For completeness we give the proof of Theorem \ref{thm:gibbs-prob} in Appendix \ref{sec:appa}.

\subsection{Random walk bridge estimates}

As seen above, the point-to-line profile can be described as a random walk bridge, conditioned on staying above a certain function.
To prove Proposition \ref{prop:multiple-peak}, in this subsection we study such random walk bridges, on the probability for the sum of two having multiple peaks, and each staying above a `well-behaved' function.

Fix arbitrary $C_0>0$, and in this subsection the constants $c,C$ would depend on $C_0$, and the values can change from line to line. We define random walk bridges $\uS, \oS$ as following.
Take any $g>0$ and $n,m\in\N$, and $\us, \os \in \R$, satisfying that
\begin{enumerate}
    \item $|\us|,|\os|<C_0n$ and $|\us+\os|<\sqrt{g}$.
    \item  $n^{0.9}<g$, $gm<n$, and $\log(n)^2<m<n^{0.1}$.
\end{enumerate}

Let $\uS, \oS :\llbracket 0, n\rrbracket \to \R$ be two independent copies of $S$ on $\llbracket 0, n\rrbracket$, conditioned on that $\uS(n)=\us$ and $\oS(n)=\os$.

Such random walk bridges could also be defined in the following alternative way.
Denote $\alpha= \os/n$, then we have $|\alpha| < C_0$. Take $\beta \in \R$ such that $\frac{8\beta}{1-4\beta^2}=\alpha$.
Let $X$ be a random variable with probability density given by $\P[X\in [x,x+\rd x)]=\frac{1-4\beta^2}{4}e^{-|x-\alpha|/2-\beta x}\rd x$; then we have $\E[X] =0$.
We consider a random walk $\oS':\Z_{\ge 0}\to \R$, where each $\oS'(i)-\oS'(i-1)$ has the same distribution as $X$; and we let $\uS'$ be an independent copy of $-\oS'$.
Then $\{\oS(i)-\alpha i\}_{i=0}^n$ is $\oS'$ on $\llbracket 0, n\rrbracket$ conditioned on that $\oS'(n)=0$;
and $\{\uS(i)+\alpha i\}_{i=0}^n$ is $\uS'$ on $\llbracket 0, n\rrbracket$ conditioned on that $\uS'(n)=\os+\us$.
We denote $\ocT^i(x)=\P[\oS'(i)\in[x,x+\rd x)]/\rd x$ and $\ucT^i(x)=\P[\uS'(i)\in[x,x+\rd x)]/\rd x$.

We would need the following bound on the probability of exhibiting multiple peaks for the sum $\uS+\oS$.
\begin{lemma}  \label{lem:S-multi-peak}
Let $\cE$ be the event where $\uS(i)+\oS(i)<\sqrt{g}$ for each $0\le i \le n$; and there exist integers $0<i_1<\ldots<i_m<n$, such that $i_{j+1}-i_j>g$ and $\oS(i_j)+\uS(i_j)>-\sqrt{g}$ for each $j$.
Then we have $\P[\cE]<Ce^{-cm}$.
\end{lemma}
\begin{proof}
We let $\cE_1$ be the event where $\uS(i)+\oS(i)<\sqrt{g}$ for each $0\le i \le \lfloor n/2\rfloor$, and there exist integers $0<i_1<\ldots<i_{\lfloor m/2\rfloor}\le \lfloor n/2\rfloor$, such that $i_{j+1}-i_j>g$ and $\oS(i_j)+\uS(i_j)>-\sqrt{g}$ for each $j$;
also let $\cE_2$ be the event where $\uS(i)+\oS(i)<\sqrt{g}$ for each $\lfloor n/2\rfloor\le i \le n$; and there exist integers $\lfloor n/2\rfloor \le i_1<\ldots<i_{\lfloor m/2\rfloor}<n$, such that $i_{j+1}-i_j>g$ and $\oS(i_j)+\uS(i_j)>-\sqrt{g}$ for each $j$.
Then we have that $\cE\subset \cE_1 \cup \cE_2$. We shall now bound $\P[\cE_1]$, and the bound for $\P[\cE_2]$ would follow similarly.

We consider the event $\cE'$, where $\uS'(i)+\oS'(i)<\sqrt{g}$ for each $0\le i \le \lfloor n/2\rfloor$, and there exist integers $0<i_1<\ldots<i_{\lfloor m/2\rfloor}\le \lfloor n/2\rfloor$, such that $i_{j+1}-i_j>g$ and $\oS'(i_j)+\uS'(i_j)>-\sqrt{g}$ for each $j$.
Denote \[\cT'(x,y)=\P[\cE',\oS'(\lfloor n/2\rfloor)\in[x, x+\rd x), \uS'(\lfloor n/2\rfloor)\in[y, y+\rd y)]/(\rd x\rd y).\]
We can then write
\[
\P[\cE_1]=\frac{\iint \cT'(x,y)\ocT^{\lceil n/2\rceil}(-x)\ucT^{\lceil n/2\rceil}(\os+\us-y) \rd x\rd y}{\ocT^n(0)\ucT^n(\os+\us)}.
\]
By a local limit theorem, we have that $\ocT^n(0), \ucT^n(\os+\us) > cn^{-1/2}$, while $\ocT^{\lceil n/2\rceil}(-x), \ucT^{\lceil n/2\rceil}(\os+\us-y) < Cn^{-1/2}$. Thus we have
\[
\P[\cE_1] < C\iint \cT'(x,y) \rd x\rd y = C\P[\cE'].
\]
It now suffices to bound $\P[\cE']$, the probability of an event on the random walk $\uS'+\oS'$. We denote $\uS'+\oS'$ as $\hS$ from now on, for simplicity of notations.
We use Skorokhod’s embedding of a random walk to a Brownian motion:
take a standard Brownian motion $B:\R_{\ge 0} \to \R$, there is a sequence of stopping times $0=\tau_0<\tau_1<\ldots <\tau_{\lfloor n/2\rfloor}$, such that $\{B(\tau_i)\}_{i=0}^{\lfloor n/2\rfloor}$ has the same distribution as $\{\hS(i)\}_{i=0}^{\lfloor n/2\rfloor}$.
For each $i\in\llbracket 1, \lfloor n/2\rfloor\rrbracket$, the random variables $\tau_i-\tau_{i-1}$ are i.i.d.; and by the exponential tail of each $\hS(i)-\hS(i-1)$ we have $\P[\tau_i-\tau_{i-1} > x]<Ce^{-cx^{1/3}}$.

We consider the following event on $B$. 
Let $t_0=0$. For each $j\in \N$, let $t_j = \inf \{t>t_{j-1}+cg, B(t)>-\sqrt{g}\}$, which is a stopping time.
Let $\cE_B'$ be the event where $B(t)<2\sqrt{g}$ for any $t\in [0,t_{\lfloor m/2\rfloor}]$.
For each $j$, conditioned on $t_j$ and $B(t_j)$, there is a positive probability that $\max_{t\in[t_j,t_j+cg]}B(t)>2\sqrt{g}$; so we have that $\P[\cE_B']<Ce^{-cm}$.

Under $\cE'\setminus \cE_B'$, one of the following events must happen:
\begin{itemize}
    \item $\cE_{spike}'$: there is some $i\in\llbracket 1, \lfloor n/2\rfloor\rrbracket$, such that $B(\tau_{i-1}), B(\tau_i)<\sqrt{g}$ and $B(t)>2\sqrt{g}$ for some $t\in (\tau_{i-1},\tau_i)$.
    \item $\cE_{narrow}'$: there is some $i\in\llbracket 0, \lfloor n/2\rfloor - \lfloor g\rfloor\rrbracket$, such that $\tau_{i+\lfloor g\rfloor}-\tau_i \le cg$.
\end{itemize}
This is because, assuming $\cE'\setminus \cE_{narrow}'$, there exist integers $0<i_1<\ldots<i_{\lfloor m/2\rfloor}\le \lfloor n/2\rfloor$, such that each $i_{j+1}-i_j>g$, so $\tau_{i_{j+1}}-\tau_{i_j}>cg$;
and each $\hS(\tau_{i_j})>-\sqrt{g}$.
Thus there is $t_{\lfloor m/2\rfloor}\le\tau_{i_{\lfloor m/2\rfloor}}$, 
and by further assuming that $\cE_{spike}'$ does not hold we have that $B(t)<2\sqrt{g}$ for any $t\in [0,t_{\lfloor m/2\rfloor}]$, i.e. $\cE_B'$ holds.

We have
\[
\P[\cE_{spike}']< \sum_{i=1}^{\lfloor n/2\rfloor} \P[\tau_i-\tau_{i-1} > \sqrt{g}] + \lfloor n/2\rfloor \P[\max_{t\in[0,\sqrt{g}]}B(t)>\sqrt{g}] < Cne^{-cg^{1/6}},
\]
and
\[
\P[\cE_{narrow}']\le \sum_{i=0}^{\lfloor n/2\rfloor - \lfloor g\rfloor} \P[\tau_{i+\lfloor g\rfloor}-\tau_i \le cg] < Cne^{-cg}.
\]
Thus we have that $\P[\cE']\le \P[\cE_B']+\P[\cE_{spike}']+\P[\cE_{narrow}']<Ce^{-cm}$ (since $n^{0.9}<g$ and $m<n^{0.1}$), and the conclusion follows.
\end{proof}

Define $R:\llbracket 0,n\rrbracket \to \R$ as $R(i)=(\log(n))^9\sqrt{i(n-i)/n}$.
The next result we need is on lower bounding the probability of  $\{\oS(i)-\alpha i\}_{i=0}^n$ staying above $R$.
\begin{lemma}  \label{lem:S-low-bd}
We have $\P[\oS(i)-\alpha i \ge R(i),\forall i\in \llbracket 0, n\rrbracket] > ce^{-C(\log(n))^{19}}$.
\end{lemma}
For its proof we need the following technical lemma.
\begin{lemma}  \label{lem:rw-bd-tr}
For any $k\in \N$, $k<Cn$, and $s\in \R$, let $\cD_{k,s}$ be the event where $\oS'(i)>-(\log(n))^2\sqrt{k}$ for each $i\in\llbracket 0, k \rrbracket$, and $\oS'(k)\in [s, s+C\sqrt{k}]$.
We then have that $\P[\cD_{k,s}]>ce^{-C(\log(n))^{18}}$, when $0\le s \le (\log(n))^9\sqrt{k}$.
\end{lemma}
\begin{proof}
We assume that $n$ is large enough, since otherwise the result follows by taking $C$ large and $c$ small enough.

When $k<C(\log(n))^{18}$, we have $\P[\cD_{k,s}]>\P[\oS'(i) \in [s,s+C\sqrt{k}],\forall i\in \llbracket 1,k\rrbracket] > ce^{-Cs-C\sqrt{k}}>ce^{-C(\log(n))^{18}}$, and the conclusion holds.

Now we assume that $k\ge C(\log(n))^{18}$.
Recall that $X$ is the random variable such that each $\oS'(i)-\oS'(i-1)$ has the same distribution as $X$.
Take $\theta\ge 0$ such that $\frac{\E[Xe^{\theta X}]}{\E[e^{\theta X}]}=\frac{s}{k}$. 
By taking Taylor expansions, and using that $\E[X]=0$, we have that for $\theta$ near zero, $|\E[Xe^{\theta X}]-\E[X^2]\theta|<C\theta^2$, and $|\E[e^{\theta X}] - 1| < C\theta^2$.
Then since $\frac{s}{k} \le \frac{(\log(n))^9}{\sqrt{k}}\le c$, and $\frac{\E[Xe^{\theta X}]}{\E[e^{\theta X}]}$ is non-decreasing in $\theta$, we have $\frac{cs}{k} < \theta< \frac{Cs}{k}$.
We consider the random variable $Y$, where \[\P[Y\in[y, y+\rd y)]=\P[X\in[y, y+\rd y)]e^{\theta y}\E[e^{\theta X}]^{-1}.\] We then have $\E[Y]=\frac{s}{k}$. Let $Y_1, Y_2, \ldots$ be an infinite sequence of independent copies of $Y$, and $S^*(i)=\sum_{j=1}^i Y_j$ for any $j\in \N$.
Then we have
\[
\P[\cD_{k,s}] \ge e^{-\theta(s+C\sqrt{k})}\E[e^{\theta X}]^k\P[0<S^*(k)-s<C\sqrt{k}, S^*(i)>-(\log(n))^2\sqrt{k}, \forall i\in\llbracket 0, k\rrbracket].
\]
Using that $|\E[e^{\theta X}] - 1| < C\theta^2$, we have
\[
e^{-\theta(s+C\sqrt{k})}\E[e^{\theta X}]^k > e^{-\theta(s+C\sqrt{k}) - Ck\theta^2} > e^{-Cs^2/k-Cs/\sqrt{k}} > ce^{-C(\log(n))^{18}},\]
where the second inequality is by $\theta< \frac{Cs}{k}$, and the third inequality is by $s \le (\log(n))^9\sqrt{k}$.

It now remains to prove $\P[0<S^*(k)-s<C\sqrt{k}, S^*(i)>-(\log(n))^2\sqrt{k}, \forall i\in\llbracket 0, k\rrbracket]>c$.
For each $i\in \llbracket 1,k\rrbracket$ we have
\[
\P[S^*(i)<-(\log(n))^2\sqrt{k}]
\le e^{-(\log(n))^2} \E[e^{-S^*(i)/\sqrt{k}}]
=
e^{-(\log(n))^2} \E[e^{-Y/\sqrt{k}}]^i.
\]
Like estimating $\E[e^{\theta X}]$ above, we have $|\E[e^{-Y/\sqrt{k}}]-(1-\E[Y]/\sqrt{k})|\le C\E[Y^2]/k < C/k$ by taking the Taylor expansion of $e^{-Y/\sqrt{k}}$.
Thus we have $\P[S^*(i)<-(\log(n))^2\sqrt{k}] < e^{-(\log(n))^2-\E[Y]\sqrt{k} + C} < Ce^{-(\log(n))^2}$.
By taking a union bound over $i$ we have $\P[S^*(i)>-(\log(n))^2\sqrt{k}, \forall i\in\llbracket 0, k \rrbracket] > \frac{7}{8}$, as $k<Cn$ and $n$ is large enough.
We also have that $\P[0<S^*(k)-s<C\sqrt{k}]>1/4$ when $C$ is large, by a local limit theorem.
Thus the conclusion follows.
\end{proof}

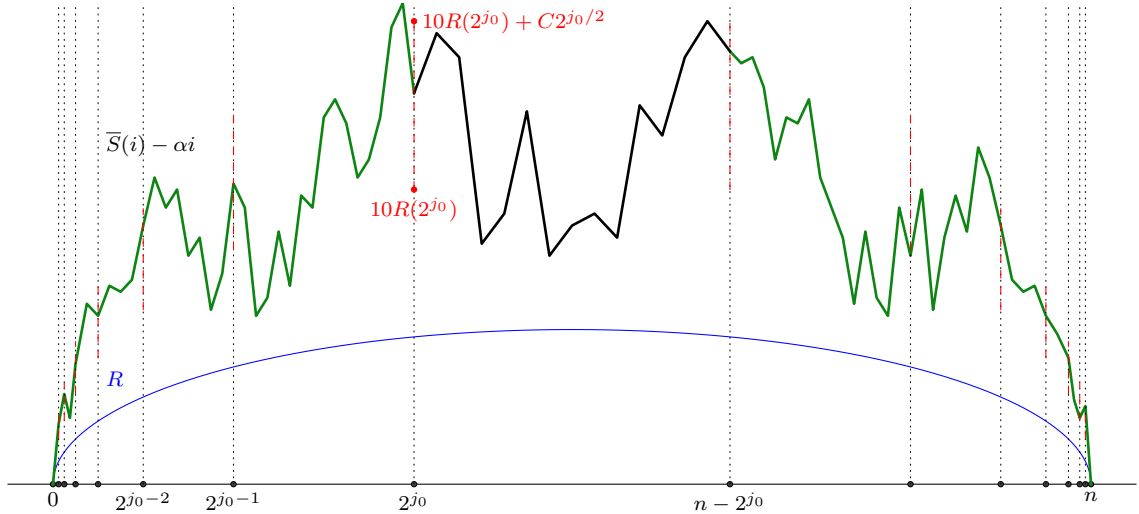
\begin{figure}[hbt!]
    \centering
\begin{tikzpicture}[line cap=round,line join=round,>=triangle 45,x=.6cm,y=.6cm]
\clip(0,-0.8) rectangle (25,8);

\draw (0,0) -- (25,0);

\draw [fill=uuuuuu] (1,0) circle (1.0pt);
\draw [fill=uuuuuu] (24,0) circle (1.0pt);
\draw [fill=uuuuuu] (1.125,0) circle (1.0pt);
\draw [fill=uuuuuu] (1.25,0) circle (1.0pt);
\draw [fill=uuuuuu] (1.5,0) circle (1.0pt);
\draw [fill=uuuuuu] (2,0) circle (1.0pt);
\draw [fill=uuuuuu] (3,0) circle (1.0pt);
\draw [fill=uuuuuu] (5,0) circle (1.0pt);
\draw [fill=uuuuuu] (9,0) circle (1.0pt);

\draw [fill=uuuuuu] (23.875,0) circle (1.0pt);
\draw [fill=uuuuuu] (23.75,0) circle (1.0pt);
\draw [fill=uuuuuu] (23.5,0) circle (1.0pt);
\draw [fill=uuuuuu] (23,0) circle (1.0pt);
\draw [fill=uuuuuu] (22,0) circle (1.0pt);
\draw [fill=uuuuuu] (20,0) circle (1.0pt);
\draw [fill=uuuuuu] (16,0) circle (1.0pt);

\draw[domain=1:24.0, samples=800, variable=\x, blue] plot ({\x}, {sqrt((\x-1)*(24-\x)/20)});

\draw [line width=1pt, darkgreen] (1,0) -- (1.125,1.) -- (1.25,1.5) -- (1.375,1.1) -- (1.5,2) -- (1.75,3) -- (2,2.8) -- (2.25,3.3) -- (2.5,3.2) -- (2.75,3.4) -- (3,4.3) -- (3.25,5.1) -- (3.5,4.6) -- (3.75,4.9) -- (4,3.8) -- (4.25,4.1) -- (4.5,2.9) -- (4.75,3.5) -- (5,5) -- (5.25,4.6) -- (5.5,2.8) -- (5.75,3.1) -- (6,4.2) -- (6.25,3.3) -- (6.5,4.8) -- (6.75,4.6) -- (7,6.1) -- (7.25,6.4) -- (7.5,6) -- (7.75,5.1) -- (8,5.4) -- (8.25,6.1) -- (8.5,7.6) -- (8.75,8.) -- (9,6.5);

\draw [line width=1pt, darkgreen] (24,0) -- (23.875,1.3) -- (23.75,1.1) -- (23.625,1.4) -- (23.5,2.1) -- (23.25,2.5) -- (23,2.8) -- (22.75,3.3) -- (22.5,3.2) -- (22.25,3.4) -- (22,4.3) -- (21.75,5.1) -- (21.5,5.6) -- (21.25,4.2) -- (21,4.8) -- (20.75,4.1) -- (20.5,2.9) -- (20.25,4.9) -- (20,3.8) -- (19.75,4.6) -- (19.5,2.8) -- (19.25,3.1) -- (19,4.2) -- (18.75,3) -- (18.5,4.1) -- (18.25,4.6) -- (18,5.1) -- (17.75,6.4) -- (17.5,6) -- (17.25,6.1) -- (17,5.4) -- (16.75,6.6) -- (16.5,7.1) -- (16.25,7.) -- (16,7.2);

\draw [line width=1pt] (9,6.5) -- (9.5,7.5) -- (10,7.1) -- (10.5,4) -- (11,4.5) -- (11.5,6.2) -- (12,3.8) -- (12.5,4.3) -- (13,4.5) -- (13.5,4.1) -- (14,6.3) -- (14.5,5.8) -- (15,7.1) -- (15.5,7.7) -- (16,7.2);

\foreach \i in {1.125,1.25,1.5,2,3,5,9,16,20,22,23,23.5,23.75,23.875}
{
\draw [dotted] (\i,0) -- (\i,10);
\draw [dashed, red] (\i,{sqrt((\i-1)*(24-\i)/5)}) -- (\i,{sqrt((\i-1)*(24-\i)/2)});
}

\begin{scriptsize}
\draw (1,0) node[anchor=north]{$0$};
\draw (24,0) node[anchor=north]{$n$};

\draw (3,0) node[anchor=north]{$2^{j_0-2}$};
\draw (5,0) node[anchor=north]{$2^{j_0-1}$};
\draw (9,0) node[anchor=north]{$2^{j_0}$};
\draw (16,0) node[anchor=north]{$n-2^{j_0}$};

\draw (2,6) node[anchor=north west]{$\oS(i)-\alpha i$};
\draw (2,2) node[anchor=north west, color=blue]{$R$};

\draw (9,4.9) node[anchor=north, color=red]{$10R(2^{j_0})$};
\draw (9,7.7) node[anchor=west, color=red]{$10R(2^{j_0})+C2^{j_0/2}$};
\end{scriptsize}

\draw [color=red, fill=red] (9,7.7) circle (1.0pt);
\draw [color=red, fill=red] (9,4.9) circle (1.0pt);

\end{tikzpicture}
\caption{An illustration of the proof of Lemma \ref{lem:S-low-bd}: the random walk bridge $\oS(i)-\alpha i$ stays above $R$, and passes through each red segment. We also work on the domains $\llbracket 0, 2^{j_0}\rrbracket$, $\llbracket 2^{j_0}, n-2^{j_0}\rrbracket$, and $\llbracket n-2^{j_0}, n\rrbracket$ separately.}
\label{fig:rwb-esti}
\end{figure}

\begin{proof}[Proof of Lemma \ref{lem:S-low-bd}]
The general idea is to do a dyadic decomposition of $\llbracket 0, n \rrbracket$ from both ends, and apply Lemma \ref{lem:rw-bd-tr}.
We assume that $n$ is large enough, since otherwise the statement holds obviously.
Let $j_0$ be the largest integer with $2^{j_0+2}<n$, and assume that $j_0>1$. We then define several events on $\oS'$:
\begin{enumerate}
    \item Let $\cD_l$ denote the following event: for each $i\in\llbracket 0, 2^{j_0}\rrbracket$ there is $\oS'(i)\ge R(i)$; and for each $0\le j \le j_0$ there is $10R(2^j)<\oS'(2^j)<10R(2^j)+C2^{j/2}$.
    \item Similarly, we denote $\cD_r$ as the event where $\oS'(i)\le -R(i)$ for each $i\in\llbracket 0, 2^{j_0}\rrbracket$, and $-10R(2^j)-C2^{j/2}<\oS'(2^j)<-10R(2^j)$ for each $0\le j \le j_0$.
    \item Denote $\cD_{c}$ as the event where there exists some $i\in \llbracket 0, n-2^{j_0+1}\rrbracket$ such that $\oS'(i)<-5R(2^{j_0}+i)$.
\end{enumerate}
We can now write
\begin{multline}\label{eq:S-low-bd-pf1}
\P[\oS(i)-\alpha i \ge R(i),\forall i\in \llbracket 0, n\rrbracket]
> \iint  \P[\cD_c^c, \oS'(n-2^{j_0+1})+x+y\in [0, \rd z) ]/\rd z \\
\times \P[\cD_l, \oS'(2^{j_0})\in [x, x+\rd x)]
\P[\cD_r, \oS'(2^{j_0})\in [y,y+\rd y) ].
\end{multline}
We note that for the integral in the right hand side, it is actually over $x,y$ with $x, -y\in (10R(2^{j_0}), 10R(2^{j_0})+C2^{j_0/2})$, and $|x+y|<C2^{j_0/2}$.
For such $x,y$, by a local limit theorem 
(and recall that $\ocT^i$ is the $i$-step transition probability of $\oS'$) we have
\begin{equation}  \label{eq:S-low-bd-pf2}
\P[\oS'(n-2^{j_0+1})+x+y\in [0, \rd z)]/\rd z = \ocT^{n-2^{j_0+1}}(-x-y) > cn^{-1/2}.
\end{equation}
We also write
\begin{multline*}
\P[\cD_c, \oS'(n-2^{j_0+1})+x+y\in [0, \rd z)]/\rd z
\le
\int \P[\cD_{c,l}, \oS'(\lfloor n/2-2^{j_0}\rfloor)\in [z, z+\rd z)] \ocT^{\lceil n/2-2^{j_0}\rceil}(-x-y-z) 
\\
+\int \P[\cD_{c,r}, \oS'(\lceil n/2-2^{j_0}\rceil)+x+y\in[z, z+\rd z)] \ocT^{\lfloor n/2-2^{j_0}\rfloor}(-z),
\end{multline*}
where $\cD_{c,l}$ denotes the event where there exists some $i\in \llbracket 0, \lfloor n/2-2^{j_0}\rfloor\rrbracket$ such that $\oS'(i)<-5R(2^{j_0}+i)$, and $\cD_{c,r}$ denotes the event where there exists some $i\in \llbracket 0, \lceil n/2-2^{j_0}\rceil\rrbracket$ such that $-x-y-\oS'(i)<-5R(n-2^{j_0}-i)$.
For each $i\in\llbracket 1, \lfloor n/2-2^{j_0}\rfloor\rrbracket$, we have
\[
\P[\oS'(i)<-5R(2^{j_0}+i)]
\le e^{-5R(2^{j_0}+i)/\sqrt{n}} \E[e^{-\oS'(i)/\sqrt{n}}]
= e^{-5R(2^{j_0}+i)/\sqrt{n}} \E[e^{-X/\sqrt{n}}]^i
<Ce^{-c(\log(n))^9},
\]
and via a union bound over $i$ we get $\P[\cE_{c,l}]<Ce^{-c(\log(n))^9}$.
Similarly we have $\P[\cE_{c,r}]<Ce^{-c(\log(n))^9}$.
By a local limit theorem we have $\ocT^{\lceil n/2-2^{j_0}\rceil}(-x-y-z), \ocT^{\lfloor n/2-2^{j_0}\rfloor}(-z) < Cn^{-1/2}$ for any $z\in\R$.
We then conclude that $\P[\cD_c, \oS'(n-2^{j_0+1})+x+y\in [0, \rd z)]/\rd z<Ce^{-c(\log(n))^9}$.
By plugging this and \eqref{eq:S-low-bd-pf2} into \eqref{eq:S-low-bd-pf1} we get
\[
\P[\oS(i)-\alpha i \ge R(i),\forall i\in \llbracket 0, n\rrbracket] > cn^{-1/2} \P[\cD_l]\P[\cD_r].
\]
By applying Lemma \ref{lem:rw-bd-tr} repeatedly, we have $\P[\cD_l]>ce^{-Cj_0(\log(n))^{18}} > ce^{-C(\log(n))^{19}}$; and by similar arguments we can have $\P[\cD_r] > ce^{-C(\log(n))^{19}}$. By plugging these bounds into the previous expression, the conclusion follows.
\end{proof}

\subsection{Fluctuation of the point-to-line profile}
In this subsection we prove Proposition \ref{prop:multiple-peak}, using the Gibbs property (Lemma \ref{lem:T-Gibbs}) and the estimates given by Lemma \ref{lem:S-multi-peak}, \ref{lem:S-low-bd}.

\begin{proof}[Proof of Proposition \ref{prop:multiple-peak}]
In this proof the constants $c,C>0$ would depend on $\psi$, and (as usual) the values can change from line to line.

We consider the following two events.
We let $\MP_1$ be the event where there exist $m$ points $w_1, \ldots, w_m \in \L_h$, such that
\begin{enumerate}
    \item $ad(w_{i+1})-ad(w_i)>g$, for each $i=1,\ldots, m-1$,
    \item $T_{\boo,w_i}+T_{w_i+(1,0),u} > \max_{w\in \L_h}T_{\boo,w}+T_{w+(1,0),u}-\sqrt{g}$, for each $i=1,\ldots, m$.
\end{enumerate}
Similarly, we let $\MP_2$ be the event where there exist $m$ points $w_1, \ldots, w_m \in \L_h$, such that
\begin{enumerate}
    \item $ad(w_{i+1})-ad(w_i)>g$, for each $i=1,\ldots, m-1$,
    \item $T_{\boo,w_i}+T_{w_i+(0,1),u} > \max_{w\in \L_h}T_{\boo,w}+T_{w+(0,1),u}-\sqrt{g}$, for each $i=1,\ldots, m$.
\end{enumerate}
We then have that $\MP(\boo,u;h,g,2m)\subset \MP_1\cup \MP_2$.
By symmetry it suffices to bound $\P[\MP_1]$.

For each $l<r \in \Z$, we denote $\MP[l,r]$ as the event where there exists $(l,h-l)=w_1, \ldots, w_m=(r,h-r) \in \L_h$, satisfying the two conditions in defining the event $\MP_1$.
We let $\cF$ be the $\sigma$-algebra, generated by $\{W^2_{\boo,v}\}_{v \in \L_{h+1}}$, $T_{\boo,(l,h-l)}$, $T_{\boo,(r,h-r)}$, and $\{W^2_{v,u}\}_{v\in \L_h}$, $T_{(l+1,h-l),u}$, $T_{(r+1,h-r),u}$.
Note that $\{W^2_{\boo,v}\}_{v \in \L_{h+1}}$ is determined by $\{\xi_v:d(v)\le h\}$, and $\{W^2_{v,u}\}_{v\in \L_h}$ is determined by $\{\xi_v:d(v)\ge h+1\}$; thus $\{W^2_{\boo,v}\}_{v \in \L_{h+1}}$, $T_{\boo,(l,h-l)}$, $T_{\boo,(r,h-r)}$ are independent of $\{W^2_{v,u}\}_{v\in \L_h}$, $T_{(l+1,h-l),u}$, $T_{(r+1,h-r),u}$.

Denote $\uA(i)=\frac{(r-i)T_{\boo,(l,h-l)} + (i-l)T_{\boo,(r,h-r)}}{r-l}$,
we then let $\ucG$ be the event where \[W^2_{\boo,(l+i,h+1-l-i)}\vee W^2_{\boo,(l+i+1,h-l-i)}<\uA(i)+(\log(n))^8 \sqrt{i(r-l-i)/(r-l)},\quad \forall i \in \llbracket 0,r-l\rrbracket .\]
Also denote $\oA(i)=\frac{(r-i)T_{(l+1,h-l),u} + (i-l)T_{(r+1,h-r),u}}{r-l}$, and let $\ocG$ be the event where \[W^2_{(l+i,h-l-i),u}\vee W^2_{(l+i+1,h-l-i-1),u}<\oA(i)+(\log(n))^8 \sqrt{i(r-l-i)/(r-l)},\quad \forall i \in \llbracket 0,r-l\rrbracket .\]
Let $\cE_1$ be the event where $|T_{\boo,(r,h-r)}-T_{\boo,(l,h-l)}+T_{(r+1,h-r),u}-T_{(l+1,h-l),u}|<\sqrt{g}$, and $\cE_2$ be the event where 
$|T_{\boo,(r,h-r)}-T_{\boo,(l,h-l)}|, |T_{(r+1,h-r),u}-T_{(l+1,h-l),u}|<C(r-l)$.
Note that $\ucG, \ocG, \cE_1, \cE_2$ are all $\cF$ measurable, and their definitions rely on $l,r$.

Recall that $S:\Z_{\ge 0}\to \R$ is the random walk where each step $S(i)-S(i-1) \sim \LAP(0,2)$.
Let $\uT, \oT :\llbracket 0, r-l\rrbracket \to \R$ be two independent copies of $S$ on $\llbracket 0, r-l\rrbracket$, conditioned on that $\uT(r-l)=T_{\boo,(r,h-r)}-T_{\boo,(l,h-l)}$ and $\uT(r-l)=T_{(r+1,h-r),u}-T_{(l+1,h-l),u}$.
Let $\cE_T$ be the event where $\uT(i)+\oT(i)<\sqrt{g}$ for each $0\le i \le r-l$, and there exist integers $0=i_1<\ldots<i_m=r-l$, such that $i_{j+1}-i_j>g$ and $\oT(i_j)+\uT(i_j)>-\sqrt{g}$ for each $j$.
Via Lemma \ref{lem:T-Gibbs}, we have
\[
\begin{split}
\P[\MP[l,r] | \cF] < & \don[\cE_1]\P[\cE_T  | \cF]\\
&\times
\P[\uT(i)\ge W^2_{\boo,(l+i,h+1-l-i)}\vee W^2_{\boo,(l+i+1,h-l-i)},\forall 0\le i \le r-l | \cF]^{-1}
\\
&\times
\P[\oT(i)\ge W^2_{(l+i,h-l-i),u}\vee W^2_{(l+i+1,h-l-i-1),u}, \forall 0\le i \le r-l | \cF]^{-1}.
\end{split}
\]
We now assume that $r-l<(\log(n))^2h^{2/3}$.
By Lemma \ref{lem:S-multi-peak}, we have $\don[\cE_1\cap\cE_2]\P[\cE_T  | \cF]<Ce^{-cm}$.
By Lemma \ref{lem:S-low-bd}, we have
\[
\don[\cE_1\cap\cE_2\cap\ucG]\P[\uT(i)\ge W^2_{\boo,(l+i,h+1-l-i)}\vee W^2_{\boo,(l+i+1,h-l-i)},\forall 0\le i \le r-l | \cF]^{-1} < Ce^{C(\log(n))^{19}},
\]
and
\[
\don[\cE_1\cap\cE_2\cap\ocG]\P[\oT(i)\ge W^2_{(l+i,h-l-i),u}\vee W^2_{(l+i+1,h-l-i-1),u}, \forall 0\le i \le r-l | \cF]^{-1} < Ce^{C(\log(n))^{19}}.
\]
Thus we conclude that $\don[\cE_2\cap \ucG\cap\ocG]\P[\MP[l,r] | \cF] <  Ce^{-cm}$. 

We now take $p$ as the largest integer with $\frac{p}{h-p}\le \frac{u_1}{u_2}$.
Let $\cE_{devi}$ be the event where there exist $p-\log(n)h^{2/3}/2 < i < j <p+\log(n)h^{2/3}/2$, such that
\[
|(T_{\boo,(i,h-i)} - D_{(i,h-i)} )-(T_{\boo,(j,h-j)}- D_{(j,h-j)})| > (\log(n))^7 \sqrt{j-i},
\]
or
\[
|(T_{(i+1,h-i),u} - D_{u-(i+1,h-i)} )-(T_{(j+1,h-j),u} - D_{u-(j+1,h-j)})| > (\log(n))^7 \sqrt{j-i}.
\]
Note that $\cE_{devi}^c$ implies $\cE_2, \ucG, \ocG$ for each $l<r$ with $p-l,r-p<\log(n)h^{2/3}/2$. 
Let $\cE_{trans}$ be the event where 
\[
\max_{|i-p|\ge \log(n)h^{2/3}/2}
T_{\boo,(i,h-i)}+T_{(i+1,h-i),u} > \max_{w\in \L_h}T_{\boo,w}+T_{w+(1,0),u}-\sqrt{g}.
\]
By summing over all pairs of $l<r$ with $p-l,r-p<\log(n)h^{2/3}/2$, and taking the expectation over $\cF$, we have
\[
\P[\MP_1]< \P[\cE_{devi}] + \P[\cE_{trans}] + Ce^{-cm}.
\]
Finally, for any $i\in\Z$ with $|i-p|\ge \log(n)h^{2/3}/2$, there is $D_{i,h-i}+D_{u-(i+1,h-i)} < D_{p,h-p}+D_{u-(p+1,h-p)} - c(\log(n))^2n^{1/3}$.
By Theorem \ref{t:onepoint} and a union bound, we get $\P[\cE_{trans}]<Ce^{-c(\log(n))^3}$.
By Lemma \ref{lem:lpp-devi} we have $\P[\cE_{devi}]<Ce^{-c(\log(n))^2}$. Thus the conclusion follows.
\end{proof}

\bibliography{slowbond}
\bibliographystyle{plain}

\newpage

\appendix

\section{Gibbs property via RSK correspondence}  \label{sec:appa}

In this appendix we prove Theorem \ref{thm:gibbs-prob}.
We shall prove an analog of it for LPP with geometric weights, using the RSK correspondence, and then pass to a scaling limit.

Fix $\beta\in (0,1)$, and we consider i.i.d. geometric random variables $\{\hxi_u\}_{v\in\Z^2}$ with odds $\beta$.
For any $u\preceq v$ and $k\in \N$ we let $\hW_{u,v}^k$ be the analog of $W_{u,v}^k$ for these $\{\hxi_u\}_{v\in\Z^2}$; and we denote $\hW_{v}^k=\hW_{\boo,v}^k$ when $\boo \preceq v$. 

Let $\sT=\Z_{\ge 0}^{\{(a,b):a\ge 0, b\ge 0, a+b\le n+1\}}$. Let $\sI\subset \Z_{\ge 0}^{I_n\cup I_{n+1}}$ be the set consisting of all $\{p_{(a,b)}^k\}_{(a,b,k)\in I_n \cup I_{n+1}}$ satisfying $p_{(a,b)}^k\le p_{(a+1,b)}^k$ and $p_{(a,b)}^k\ge p_{(a+1,b)}^{k+1}$, for any 
$(a,b,k), (a+1,b,k)\in I_n \cup I_{n+1}$;
and $p_{(a,b)}^k\le p_{(a,b+1)}^k$ and $p_{(a,b)}^k\ge p_{(a,b+1)}^{k+1}$, for any 
$(a,b,k), (a,b+1,k)\in I_n \cup I_{n+1}$.
We let $F:\sT\to\sI$ be the function such that $F:\{\hxi_{(a,b)}\}_{a\ge 0, b\ge 0, a+b\le n+1} \mapsto \{\hW_{(a,b)}^k\}_{(a,b,k)\in I_n \cup I_{n+1}}$.
\begin{lemma}  \label{lem:rsk-bi}
The map $F$ is a bijection between $\sT$ and $\sI$.
\end{lemma}
From this lemma, and noting that $\sum_{(a,b,k)\in I_{n+1}} \hW_{(a,b)}^k - \sum_{(a,b,k)\in I_{n}} \hW_{(a,b)}^k = \sum_{a\ge 0, b\ge 0, a+b\le n+1}\hxi_{(a,b)}$, we get the distribution function for $\{\hW_{(a,b)}^k\}_{(a,b,k)\in I_n \cup I_{n+1}}$.
\begin{theorem}  \label{thm:gibbs-prob-geo}
For any non-negative $\{p_{(a,b)}^k\}_{(a,b,k)\in I_n \cup I_{n+1}}\in \sI$,
we have
\[
\P\left[ \hW_{(a,b)}^k = p_{(a,b)}^k,\; \forall (a,b,k)\in I_n\cup I_{n+1} \right] = (1-\beta)^{(n+2)(n+3)/2} \beta^{\sum_{(a,b,k)\in I_{n+1}} p_{(a,b)}^k - \sum_{(a,b,k)\in I_{n}} p_{(a,b)}^k}.
\]
\end{theorem}
By sending $\beta\to 0$ and rescaling, this implies Theorem \ref{thm:gibbs-prob}.
Now it remains to prove Lemma \ref{lem:rsk-bi}.
\begin{proof}[Proof of Lemma \ref{lem:rsk-bi}]
For each $A\subset \Z_{\ge 0}^2$, let $I[A]=\{(a,b,k):(a,b)\in A, k \in \llbracket 1,1+a\wedge b\rrbracket\}$, and let $\sI[A]$ denote the set consisting of all $\{p_{(a,b)}^k\}_{(a,b,k)\in I[A]}$ satisfying the interlacing condition, i.e., $p_{(a,b)}^k\le p_{(a+1,b)}^k$ and $p_{(a,b)}^k\ge p_{(a+1,b)}^{k+1}$, for any 
$(a,b,k), (a+1,b,k)\in I[A]$; and $p_{(a,b)}^k\le p_{(a,b+1)}^k$ and $p_{(a,b)}^k\ge p_{(a,b+1)}^{k+1}$, for any $(a,b,k), (a,b+1,k)\in I[A]$.

Take $i\in\llbracket 0,n+1\rrbracket$, let $B_i=\{(a,i):0\le a \le n+1-i\}\cup \{(a,b):a+b\in\{n,n+1\}, 0\le b \le i\}$.
Then we would have $B_{n+1}=I_n\cup I_{n+1}$, and $\sI[B_{n+1}]=\sI$.
Also let $\sT_i=\Z_{\ge 0}^{\{(a,b):a\ge 0, b \ge i+1, a+b\le n+1\}}$, where we regard $\sT_{n+1}=\{0\}$.

We define a sequence of functions $F_i: \sT\to \sT_i\times \sI[B_i]$, for each $0\le i \le n+1$, by
\begin{equation} \label{eq:app-def-f}
F_i:\{\hxi_{(a,b)}\}_{a\ge 0, b\ge 0, a+b\le n+1} \mapsto \left( \{\hxi_{(a,b)}\}_{a\ge 0, b\ge i+1, a+b\le n+1}, \{\hW_{(a,b)}^k\}_{(a,b,k)\in J_i} \right).
\end{equation}
We would inductively prove that each of them is a bijection.
For $F_0$, it is a bijection since that for any $a\in\llbracket 0, n+1\rrbracket$, we have $\hW_{(a,0)}^1 = \sum_{a'=0}^a \hxi_{(a',0)}$.
Now suppose that $F_i$ is bijection. We will show that the function $F_{i+1} \circ F_i^{-1}$ is also a bijection.

Denote $B_i'=\{(a,i):0\le a \le n-i\}$ and $B_i''=\{(a,i+1):0\le a \le n-i\}\cup\{(n-i,i)\}$. We now define a bijection between $\Z_{\ge 0}^{\{(a,i+1)\}_{a=0}^{n-i}} \times \sI[B_i']$ and $\sI[B_i'']$.

\noindent\textbf{The function $G_i : \Z_{\ge 0}^{\{(a,i+1)\}_{a=0}^{n-i}} \times \sI[B_i'] \to \sI[B_i'']$.}

Take $\{q_{(a,i+1)}\}_{a=0}^{n-i}\in \Z_{\ge 0}^{\{(a,i+1)\}_{a=0}^{n-i}}$ and $\{p_{(a,i)}^k\}_{(a,i,k)\in I[B_i']} \in \sI[B_i']$.
We define $q_{(a,b)}$ for $(a,b)\in \llbracket 0,n-i\rrbracket \times \llbracket 0, i\rrbracket$:
when $a+b \le i-1$ we let $q_{(a,b)} = 0$; let $q_{(a,i-a)}=p_{(a,i-a)}^{a+1}$ for any $a\in \llbracket 0, (n-i)\wedge i\rrbracket$; and let $q_{(a,b)}=p_{(a,i)}^{i+1-b} - p_{(a-1,i)}^{i+1-b}$ for $a+b\ge i+1$.

For each $(a,b,k) \in I[\llbracket 0,n-i\rrbracket \times \llbracket 0, i+1\rrbracket]$ we let $l_{(a,b)}^k=\max_{\gamma_1,\ldots, \gamma_k} \sum_{j=1}^k \sum_{u\in \gamma_j}q_u$, where the maximum is over all mutually disjoint up-right paths $\gamma_1,\ldots, \gamma_k$ contained in $\llbracket 0,a\rrbracket \times \llbracket 0, b\rrbracket$; and let $p_{(a,b)}^1=l_{(a,b)}^1$, and $p_{(a,b)}^k=l_{(a,b)}^k-l_{(a,b)}^{k-1}$ for $k\ge 2$. 
Then we recover $p_{(a,i)}^k$ for each $(a,i,k)\in I[B_i']$.
This is because, from our construction of $\{q_{(a,b)}\}_{(a,b)\in \llbracket 0,n-i\rrbracket \times \llbracket 0, i+1\rrbracket}$ and the interlacing condition, for any $(a,i,k)\in I[B_i']$, the $k$ mutually disjoint maximum paths can be taken as the $k$ rows in $\llbracket 0,a\rrbracket \times \llbracket i+1-k, i \rrbracket$.

Obviously $\{p_{(a,b)}^k\}_{(a,b,k)\in I[B_i'']}$ satisfies the interlacing condition and is in $\sI[B_i'']$.
We then define \[G_i:\left(\{q_{(a,i+1)}\}_{a=0}^{n-i}, \{p_{(a,b)}^k\}_{(a,b,k)\in I[B_i']}\right)\mapsto \{p_{(a,b)}^k\}_{(a,b,k)\in I[B_i'']}.\]

Denote $B_i^+=\llbracket 0,n-i\rrbracket \times \{i+1\} \cup \{n-i\}\times \llbracket 0, i+1\rrbracket$.
The array $\{p_{(a,b)}^k\}_{(a,b,k)\in I[B_+]}$ is also equivalent to a pair of semi-standard Young tableaux $U, V$ of the same shape, obtained by applying RSK correspondence to $\{q_{(a,b)}\}_{(a,b)\in \llbracket 0,n-i\rrbracket \times \llbracket 0, i+1\rrbracket}$.
For each $(a,b)\in \llbracket 0,n-i\rrbracket \times \llbracket 0, i+1\rrbracket$ there is a Young diagram of at most $1+a\wedge b$ rows, denoted as $Y_{(a,b)}$, where there are $p_{(a,b)}^k$ boxes in row $k$. 
The Young tableaux $U, V$ would have shape $Y_{(n-i,i+1)}$.
For each box in $Y_{(n-i,i+1)}$, if $j$ is the smallest number such that the box is also in $Y_{(j,i+1)}$, then we write $j+1$ in this box in $U$;
if $j'$ is the smallest number such that the box is also in $Y_{(n-i,j')}$, then we write $j'+1$ in this box in $V$ (see e.g. \cite{stanley1999enumerative}).
We can also define a Young tableaux $U'$ with shape $Y_{(n-i,i)}$.
For each box in $Y_{(n-i,i)}$, if $j$ is the smallest number such that the box is also in $Y_{(j,i)}$, then we write $j+1$ in this box for $U'$.
This $U'$ is also (the first Young tableaux in) the RSK correspondence of $\{q_{(a,b)}\}_{(a,b)\in \llbracket 0,n-i\rrbracket \times \llbracket 0, i\rrbracket}$.
It is known that from $U'$, we can get $U$ by doing $q_{(a,i+1)}$ times Schensted insertions of $a+1$, sequentially for $a=0,\ldots, n-i$.

\noindent\textbf{The function $H_i : \sI[B_i''] \to \Z_{\ge 0}^{\{(a,i+1)\}_{a=0}^{n-i}} \times \sI[B_i']$.}

Take any $\{p_{(a,b)}^k\}_{(a,b,k)\in I[B_i'']} \in \sI[B_i'']$. For any $b\in\llbracket 0, i-1\rrbracket$, take
\begin{equation}  \label{eq:app-p}
p_{(n-i,b)}^k = 
\begin{cases}
p_{(n-i,i)}^{k+i-b}, \;& k+i-b \le 1+(n-i)\wedge i ; \\
0, \; & k+i-b > 1+(n-i)\wedge i.
\end{cases}    
\end{equation}
The array $\{p_{a,b}^k\}_{(a,b,k)\in I[B_i^+]}$ is in $\sI[B_i^+]$ and can be written as a pair of Young tableaux of the same shape.
Applying the inverse of RSK we get non-negative integers $\{q_{(a,b)}\}_{(a,b)\in \llbracket 0,n-i\rrbracket \times \llbracket 0, i+1\rrbracket}$. For $(a,i,k)\in I[B_i']$ we let $p_{(a,i)}^k=\sum_{j=0}^a q_{(j,i+1-k)}$.

To define $H_i$ there are several things we need to verify. First, note that $B_i''\cap B_i'=\{(n-i,i)\}$, so we verify that we recover $\{p_{(n-i,i)}^k\}_{k=1}^{1+(n-i)\wedge i}$. Second, we verify $\{p_{(a,b)}^k\}_{(a,b,k)\in I[B_i']} \in \sI[B_i']$. 
Define $l_{(n-i,b)}^k=\max_{\gamma_1,\ldots, \gamma_k} \sum_{j=1}^k \sum_{u\in \gamma_j}q_u$, where the maximum is over all $k$ mutually disjoint up-right paths contained in $\llbracket 0,n-i\rrbracket \times \llbracket 0, b\rrbracket$.
Then (due to the inverse of RSK) we have $l_{(n-i,b)}^k=\sum_{j=1}^k p_{(n-i,b)}^j$.
From \eqref{eq:app-p} we have that $l_{(n-i,b-1)}^b + l_{(n-i,b+k-1)}^k =l_{(n-i,b+k-1)}^{b+k} $, for any $0< b<b+k \le i+1$. This implies that $l_{(n-i,b+k-1)}^k = \sum_{a=0}^{n-i}\sum_{b'=b}^{b+k-1} q_{(a,b')}$,
for any $0\le b<b+k \le i+1$.
Thus $\sum_{j=0}^{n-i} q_{(j,i+1-k)}=l_{(n-i,i)}^k-l_{(n-i,i)}^{k-1}$ and recovers $p_{(n-i,i)}^k$.
Also we get that for any $a\in\llbracket 0, n-i \rrbracket$ and $b\in\llbracket 1, i \rrbracket$, there is $\sum_{a'=0}^{a-1} q_{(a',b)} \ge \sum_{a'=0}^{a} q_{(a',b-1)}$, implying that $\{p_{(a,i)}^k\}_{(a,i,k)\in I[B_i']} \in \sI[B_i']$.

We then define \[H_i: \{p_{(a,b)}^k\}_{(a,b,k)\in I[B_i'']} \mapsto \left(\{q_{(a,i+1)}\}_{a=0}^{n-i}, \{p_{(a,b)}^k\}_{(a,b,k)\in I[B_i']}\right).\]
It is straight forward to check that $G_i\circ H_i$ and $H_i\circ G_i$ are identity maps, so $G_i$ is a bijection.

Finally, for the function $F_{i+1}\circ F_i^{-1}$ from $\sT_i\times \sI[B_i]$ to $\sT_{i+1}\times \sI[B_{i+1}]$, it is in fact $G_i$ on $\Z_{\ge 0}^{\{(a,i+1)\}_{a=0}^{n-i}} \times \sI[B_i']$, because it can be realized as the same sequence of Schensted insertions on these coordinates; and it is the identity map on other coordinates.
Thus $F_{i+1}$ is also a bijection, and the induction closes. By taking $F=F_{n+1}$ the conclusion follows.
\end{proof}
\end{document}